\def\coralreport{1}

\documentclass[11pt]{article}

\usepackage{ifthen}

\ifthenelse{\coralreport = 1}{
\usepackage{isetechreport}
\def\reportyear{24T}
\def\reportno{017}
\def\revisionno{0}
\def\originaldate{Dec 11, 2024}
\def\revisiondate{NA}
\coraltrue
\cvcrfalse

}{}

\usepackage{graphicx}%
\usepackage{multirow}%
\usepackage{amsmath,amssymb,amsfonts}%
\usepackage{amsthm}%
\usepackage{mathrsfs}%
\usepackage[title]{appendix}%
\usepackage{xcolor}%
\usepackage{textcomp}%
\usepackage{manyfoot}%
\usepackage{booktabs}%
\usepackage{algorithm}%
\usepackage{algorithmicx}%
\usepackage{algpseudocode}%
\usepackage{listings}%
\theoremstyle{thmstyleone}%
\newtheorem{theorem}{Theorem}%

\theoremstyle{thmstyletwo}%

\theoremstyle{thmstylethree}%
\newtheorem{definition}{Definition}%
\newtheorem{assumption}{Assumption}

\newtheorem{lemma}{Lemma}[section]
\newtheorem{corollary}{Corollary}[section]

\begin{document}

\title{A Preconditioned Inexact Infeasible Quantum Interior Point Method for Linear Optimization}






\ifthenelse{\coralreport = 1}{

\author{Zeguan Wu\thanks{E-mail: \texttt{zew220@lehigh.edu}}}
\author{Xiu Yang\thanks{E-mail: \texttt{xiy518@lehigh.edu}}}
\author{Tamás Terlaky\thanks{E-mail: \texttt{tat208@lehigh.edu}}}
\affil{Department of Industrial and Systems Engineering, Lehigh University, USA}

\titlepage

}{

\author{Zeguan Wu \and Xiu Yang \and Tamás Terlaky}
\institute{Zeguan Wu\\
Department of Industrial and Systems Engineering, Lehigh University, USA\\
\email{zew220@lehigh.edu}
\and 
Xiu Yang\\
Department of Industrial and Systems Engineering, Lehigh University, USA\\
\email{xiy518@podunk.edu}
\and
Tamás Terlaky\\
Department of Industrial and Systems Engineering, Lehigh University, USA
}
\subclass{81P68, 90C05, 90C51, 65F08}

\date{\today}

}

\maketitle

\begin{abstract}
Quantum Interior Point Methods (QIPMs) have been attracting significant interests recently due to their potential of solving optimization problems substantially faster than state-of-the-art conventional algorithms. In general, QIPMs use Quantum Linear System Algorithms (QLSAs) to substitute classical linear system solvers. However, the performance of QLSAs depends on the condition numbers of the linear systems, which are typically proportional to the square of the reciprocal of the duality gap in QIPMs. To improve conditioning, a preconditioned inexact infeasible QIPM (II-QIPM) based on optimal partition estimation is developed in this work. 
We improve the condition number of the linear systems in II-QIPMs from quadratic dependence on the reciprocal of the duality gap to linear, and obtain better dependence with respect to the accuracy when compared to other II-QIPMs. Our method also attains better dependence with respect to the dimension when compared to other inexact infeasible Interior Point Methods.

\ifthenelse{\coralreport = 0}{
\bigskip\noindent
{\bf Keywords:} Quantum Computing, Linear Optimization, Interior Point Method, Preconditioning Method
}{}

\end{abstract}

\section{Introduction}\label{sec: intro}


Linear Optimization (LO) is defined as 
optimizing a linear objective function under a set of linear constraints. The two most popular families of algorithms for solving LO problems are simplex algorithms and Interior Point Methods (IPMs) \cite{bertsimas1997introduction,roos1997theory}. Despite its efficiency for many practical problems, simplex methods may take exponentially many iterations to find an optimal solution, whereas IPMs guarantee an optimal solution in polynomial number of iterations.

The modern age of IPMs were launched by Karmarkar's projective method for LO problems \cite{karmarkar1984new}. Since then, many IPM variants have been proposed and studied for not only LO problems but also nonlinear optimization problems \cite{roos1997theory,nesterov1994interior,polik2010interior}.
Contemporary IPMs start from an interior point and progress to the optimal set by moving within a neighbourhood of an analytic curve, known as the central path. Depending on whether feasibility is satisfied by the iterates, IPMs can be categorized into feasible IPMs and infeasible IPMs.
Feasible IPMs start with a strictly feasible solution and maintain feasibility at each iterations, whereas infeasible IPMs does not require feasibility to be exactly satisfied by any iterates.
To find an $\epsilon$-approximate solution for an LO problem with $n$ variables and $m$ constraints $(m\leq n)$, feasible IPMs require $\mathcal{O}\left(\sqrt{n} \log(1/\epsilon)\right)$ IPM iterations while infeasible IPMs require $\mathcal{O}\left(n \log(1/\epsilon)\right)$ IPM iterations \cite{potra1996infeasible,roos2006full}. 

In each IPM iteration, a Newton linear system needs to be solved for the search direction, which is called the Newton direction. The Newton linear system can take different forms, including full Newton system, augmented system, and normal equation system. Specifically, the full Newton system is asymmetric and can be solved exactly using Gaussian elimination with $\mathcal{O}\left(n^3\right)$ arithmetic operations.
The augmented system is symmetric but indefinite, and it can be solved exactly using Bunch–Parlett factorization with $\mathcal{O}\left(n^3\right)$ arithmetic operations \cite{bunch1971direct}. 
The normal equation system is symmetric positive definite and can be solved exactly using Cholesky decomposition with $\mathcal{O}\left(n^3\right)$ arithmetic operations. These systems can also be solved inexactly using methods including Krylov subspace methods, which may require fewer iterations if the target accuracy is moderate.
However, the Newton directions obtained in this way are inexact, which may make the IPM iterates infeasible. Thus, the resulting IPMs are infeasible IPMs and have worse iteration complexity. On the other hand, it is possible to maintain feasibility when using inexact Newton system solutions. For example, in \cite{mohammadisiahroudi2021inexact,zeguanlcqo,augustino2023quantum}, the authors proposed an orthogonal subspace system whose inexact solutions give feasible Newton directions.

The recent development of quantum computing has triggered the study of novel quantum algorithms, including hybrid quantum-classic algorithms \cite{dalzell2023quantum}. {Many hybrid algorithms use Quantum Linear System Algorithms (QLSAs) as they can solve some linear systems faster than classical linear system algorithms \cite{harrow2009quantum, childs2017quantum, chakraborty2018power}.} {For example, based on the seminal work of the QLSA \cite{harrow2009quantum}, various approaches that improve the performance of IPMs by solving the Newton systems with QLSAs are proposed \cite{kerenidis2020quantumipm, mohammadisiahroudi2022efficient, augustino2023quantum}.} 
Although QLSAs are more efficient w.r.t. dimension, there are several issues when applying them to IPMs. One issue is that QLSAs only produce quantum states, which cannot be used by classical IPMs directly. So QLSAs are always coupled with quantum state tomography algorithms (QTAs) \cite{gilyen2018quantum}, which extract classical results from the quantum results. At each iteration of these quantum IPMs (QIPMs), QLSAs are used to solve quantized Newton systems and QTAs are used to obtain classical solutions. 
{Here, QTAs have polynomial complexity with respect to the reciprocal of their precisions. Also, such precisions should be of the same magnitude as the precision of QIPMs.} As a result, although many QIPMs attain better complexities with respect to problem size than their classical counterparts, they could have worse complexity with respect to precision.
{Another issue with QLSAs is their linear (or worse) complexity dependence on the condition number of Newton systems.} The condition number of the Newton systems arising from IPMs can increase to infinity when IPMs approach optimality. Despite the improved dependence on the problem size, the dependence on condition number indicates the necessity to precondition the Newton systems arising from QIPMs.

In this work, we propose a preconditioned II-QIPM using the preconditioning method proposed by Chai and Toh in \cite{chai2007preconditioning} and the II-QIPM framework proposed by Mohammadisiahroudi et al in \cite{mohammadisiahroudi2022efficient}.
{We prove the condition number of the preconditioned linear system is improved from $\mathcal{O} \left({1}/\mu^2\right)$ to $\mathcal{O} \left({1}/\mu\right)$, where $\mu$ is the corresponding central path parameter.}
Our algorithm demonstrates $\mathcal{O}(1/\epsilon^3)$ speed-up compared with the II-QIPM proposed in \cite{mohammadisiahroudi2022efficient} when the iterative refinement method is not used.

\section{Preliminaries}\label{sec: Preliminary}
In this section, we introduce the notation used in this paper. Then we give a brief introduction to classical inexact infeasible IPMs (II-IOMs) for LO problems. Finally, we introduce the II-QIPM for LO proposed in \cite{mohammadisiahroudi2022efficient}.

\subsection{Notation}
In this work, we take the following conventions. Vectors are denoted by lowercase letters, matrices by capital letters. For a vector $v\in\mathbb{R}^n$, its $i$th element is denoted by $v_{[i]}$, i.e., $v=\left(v_{[1]}, \dots, v_{[n]}\right)^T$. For a vector with subscript, e.g., $v_1\in\mathbb{R}^n$, its $i$th element is denoted by $v_{1,[i]}$. 
For a diagonal matrix $M\in \mathbb{R}^{n\times n}$, its $i$th diagonal element is denoted by $M_{[i]}$.  For a diagonal matrix with subscript, e.g., $M_1\in\mathbb{R}^{n\times n}$, its $i$th diagonal element is denoted by $M_{1,[i]}$. 

For a matrix $M\in \mathbb{R}^{m\times n}$ with $m\leq n$, it has $m$ singular values. We assume that the singular values are sorted in non-increasing order starting from index $1$. The singular values of $M$ are denoted by
\begin{equation*}
    \sigma_1(M)\geq \sigma_2(M)\geq \cdots\geq \sigma_m(M),
\end{equation*}
except for the special notation $\sigma_0(M)$ denoting the smallest nonzero singular value of $M$. For a matrix $M\in\mathbb{R}^{m\times n}\  (m\leq n)$  with $k\ (k<m)$ nonzero singular values, we have
\begin{equation*}
    \sigma_1(M)\geq \cdots \geq \sigma_k(M) = \sigma_0(M) >\sigma_{k+1}(M)=\cdots=\sigma_m(M)=0.
\end{equation*}
The condition number of matrix $M$ is denoted by $\kappa(M)$ and is defined as
\begin{equation*}
    \begin{aligned}
        \kappa(M) = \frac{\sigma_1(M)}{\sigma_0(M)},
    \end{aligned}
\end{equation*}
which implies that, in this work, the condition number for any matrix, except zero matrices, is always finite.
For a nonsingular matrix $M\in \mathbb{R}^{(n+m)\times(n+m)}$ with the following 2-2 block structure
\begin{equation*}
    \begin{aligned}
        M=\begin{bmatrix}
        M_{11}&M_{12}\\M_{21}&M_{22}
        \end{bmatrix},
    \end{aligned}
\end{equation*}
where $M_{11}\in \mathbb{R}^{n\times n}$ and $M_{22}\in \mathbb{R}^{m\times m}$ and both nonsingular, we denote the Schur complements \cite{crabtree1969identity} of $M$ by
\begin{equation*}
    (M/M_{11}) = M_{22}- M_{21} M_{11}^{-1} M_{12} \text{ and } (M/M_{22}) = M_{11} - M_{12} M_{22}^{-1} M_{21}.
\end{equation*}
We use $e_n$ and $0_n$ to denote the $n$-dimensional all one vector and all zero vector, respectively. When the dimension is obvious from the context, we may use $e$ and $0$ for simplicity.
We use $I_{n\times m}$ and $0_{n\times m}$ to denote the $n\times m$ dimensional identity matrix and zero matrix, respectively. When the dimension is obvious from the context, we may write $I$ or $0$.
{We take the convention of big-$O$, big-Omega, and big-Theta notations for complexity. We also use $\tilde{}$ above these notations when ignoring the polylogarithmic overhead. For example, $\tilde{\mathcal{O}}_{n}(n^2) = \mathcal{O}(n^2{\rm poly}\log(n))$.}


\subsection{Inexact Infeasible IPMs for LO Problems}
In this section, we briefly review II-IPMs for LO problems. In this work, an LO problem is defined as follows. 
\begin{definition}[LO Problem]\label{LOdef: LO}
    For vectors $b\in \mathbb{R}^m,\ c\in \mathbb{R}^n$, and matrix $A\in \mathbb{R}^{m\times n}$ with ${\rm rank}(A)=m\leq n,$ we define the primal LO problem as:
    \begin{equation}\label{LOprimal}\tag{P}
        \min_{x\in \mathbb{R}^n}\  c^T x, \ {\rm s.t.\ } 
        Ax = b,\  x \geq 0
    \end{equation}
and the dual LO problem as
\begin{equation}\label{LOdual}\tag{D}
    \max_{y\in \mathbb{R}^m, s\in \mathbb{R}^n} \  b^T y, \ {\rm s.t.\ }
    A^Ty + s = c,\ s \geq 0.
\end{equation}
\end{definition}
We make the following assumptions on problems~\eqref{LOprimal} and~\eqref{LOdual}.
\begin{assumption}\label{LOassumption: integer data}
    The input data, i.e., entries of $b$, $c$, and $A$, are integers.
\end{assumption}
\begin{assumption}\label{LOassumption: interior}
There exists a solution $(x,y,s)$ such that
\begin{equation*}
    Ax=b,\ A^Ty+s=c, \text{ and } (x,s)>0.
\end{equation*}
\end{assumption}
{Assumption~\ref{LOassumption: integer data} implicitly exists in most IPMs work. In this paper, it is explicitly used to prove Theorem~\ref{LOtheorem: main}.} 
Assumption \ref{LOassumption: interior} is not restrictive because we can apply the self-dual embedding model \cite{roos1997theory} to problem~\eqref{LOprimal} and~\eqref{LOdual} to obtain an equivalent LO problem which has a strictly feasible solution.
The set of primal-dual feasible solutions can be defined as
\begin{equation*}
    \mathcal{PD} = \left\{(x,y,s)\in \mathbb{R}^n\times\mathbb{R}^m\times\mathbb{R}^n:\ Ax=b,\ A^Ty + s = c,\ (x,s)\geq 0\right\}
\end{equation*}
and the set of interior primal-dual feasible solutions can be defined as
\begin{equation*}
    \mathcal{PD}^0 = \left\{(x,y,s)\in \mathbb{R}^n\times\mathbb{R}^m\times\mathbb{R}^n:\ Ax=b,\ A^Ty + s = c,\ (x,s)> 0\right\}.
\end{equation*}
According to the Strong Duality Theorem \cite{roos1997theory}, the set of optimal solutions can be defined as
\begin{equation*}
    \mathcal{PD}^\ast = \left\{(x,y,s)\in \mathcal{PD}:XSe = 0\right\},
\end{equation*}
where $X={\rm diag} (x)$ and $S={\rm diag} (s)$.
Next, let $\epsilon>0$, and the set of $\epsilon$-approximate solutions to problem~\eqref{LOprimal} and~\eqref{LOdual} can be defined as
\begin{equation*}
    \mathcal{PD}_{\epsilon} = \left\{ (x,y,s)\in \mathbb{R}^n\times\mathbb{R}^m\times\mathbb{R}^n: x^Ts\leq n\epsilon,\ \left\|(r_p,r_d)\right\|_2\leq\epsilon \right\},
\end{equation*}
where $r_p = b-Ax,\ r_d = c - A^T y - s$. Under Assumption~\ref{LOassumption: interior}, for any $\mu>0$, the perturbed optimality conditions
\begin{equation}\label{LOsys: kkt}
    \begin{aligned}
        Ax      &=b,\ x\geq 0\\
        A^Ty + s&=c,\ s\geq 0\\
        XSe     &=\mu e,
    \end{aligned}
\end{equation}
have a unique solution $\left(x(\mu), y(\mu), s(\mu)\right)$ that defines the primal and dual central path
\begin{equation*}
    \mathcal{CP} = \left\{ (x,y,s)\in \mathcal{PD}^0\ | \ XSe= \mu e,\ \forall \mu >0  \right\}.
\end{equation*}
IPMs apply Newton's method to solve system~\eqref{LOsys: kkt} approximately. At each iteration of inexact infeasible IPMs (II-IPMs), a candidate solution to problem~\eqref{LOprimal} and~\eqref{LOdual} is updated by moving along the solution to the full Newton system (NS), i.e., {by solving}
\begin{equation}\label{LOsys: Newton system}\tag{NS}
    \begin{aligned}
    \begin{bmatrix}
    A & 0 & 0\\
    0 & A^T & I\\
    S &0 &X
    \end{bmatrix}
    \begin{bmatrix}
    \Delta x\\
    \Delta y\\
    \Delta s
    \end{bmatrix}=\begin{bmatrix}
    r_p\\r_d\\r_c
    \end{bmatrix},
\end{aligned}
\end{equation}
where $r_c = \beta_1 \mu e - X Se$ and $\beta_1\in(0,1)$ is the central path parameter reduction factor. {Here,} $(r_p,\ r_d,\ r_c)$ are defined at each iteration but for simplicity here we omit the superscript for iteration index.
\eqref{LOsys: Newton system} can be reduced into the so-called augmented system (AS)
\begin{align}\label{LOsys: augmented system}\tag{AS}
    \begin{bmatrix}
    -D &A^T\\ A & 0
    \end{bmatrix}
    \begin{bmatrix}
    \Delta x\\\Delta y
    \end{bmatrix}=\begin{bmatrix}
    r_{dc}\\r_p
    \end{bmatrix},
\end{align}
where $D= X^{-1}S$ and $r_{dc}=r_d - X^{-1}r_c$. \eqref{LOsys: augmented system} can be further reduced into the normal equation system (NES)
\begin{equation}\label{LOsys: normal equation system}\tag{NES}
    \begin{aligned}
      AD^{-1}A^T \Delta y = AD^{-1}r_{dc} + r_p.
    \end{aligned}
\end{equation}
In each II-IPM iterations, a linear solver is used to inexactly solve for the Newton {direction} $\left(\Delta x, \Delta y, \Delta s\right)$. Depending on which linear system is solved and which linear solver is used, the requirements and properties of the inexact Newton directions might be different \cite{dembo1982inexact,kelley1995iterative, al2009convergence}.
If one uses any inexact method to solve~\eqref{LOsys: Newton system}, then all the three residual terms will be nonzero, i.e., $A\Delta x - r_p\neq 0, A^T\Delta y + \Delta s -r_d\neq 0, S\Delta x + X\Delta s - r_c\neq 0$. 
These nonzero residual terms will bring more challenges in the analysis of II-IPMs.
In \cite{al2009convergence,monteiro2003convergence}, the authors used similar techniques to eliminate the residuals for the first two equations in ~\eqref{LOsys: Newton system} and to move all the residuals into the third equation. This is helpful for convergence analysis and we discuss this in more details in the next section.
Once an inexact Newton direction is obtained, the candidate solution is updated by moving along the Newton direction while staying in a neighbourhood of the central path. In this work, for any $0<\gamma_1<1$ and $1\leq\gamma_2$, we define a central path neighborhood as
\begin{equation}
    \begin{aligned}\label{LOdef: neighbourhood}
    \mathcal{N}(\gamma_1,\gamma_2)=\left\{ (x,y,s)\in\mathbb{R}^{n+m+n}:\ (x,s)\geq0,\ xs\geq\gamma_1\mu e,\ \|(r_p,r_d)\|_2\leq\gamma_2\mu \right \},
\end{aligned}
\end{equation}
where $\mu= x^Ts/n$.

\subsection{Inexact Infeasible QIPMs for LO Problems}\label{LOsection: IIQIPM LO}
%
QIPMs use quantum algorithms to solve Newton linear systems. Due to the non-negligible inaccuracy in current quantum computers, the Newton directions obtained from quantum algorithms are inexact and thus might be infeasible. In \cite{mohammadisiahroudi2022efficient}, the authors designed and analyzed the first II-QIPM for LO problems. They followed the general framework of the II-IPM proposed in \cite{al2009convergence} and  used quantum algorithms to obtain the Newton directions.
Instead, if one solves \eqref{LOsys: normal equation system} using quantum algorithms to get $\Delta y$ and sets
\begin{equation*}
    \begin{aligned}
        \Delta s &= r_d - A^T  \Delta y\\
        \Delta x &= S^{-1}\left(r_c - X\Delta s\right),
    \end{aligned}
\end{equation*}
then the second and third equation in ~\eqref{LOsys: Newton system} are satisfied exactly while the residual of the first equation in ~\eqref{LOsys: Newton system} equals to the residual of \eqref{LOsys: normal equation system}.
It was recognized in \cite{al2009convergence,dembo1982inexact} that one can move the residual from the first equation to the third, which makes the first two equations in ~\eqref{LOsys: Newton system} be satisfied exactly and makes the analysis easier.
To do so, one first needs a basis for $A$. In \cite{mohammadisiahroudi2022efficient}, the authors claimed that the cost to find such a basis is dominated by the cost of solving ~\eqref{LOsys: Newton system} and thus ignored. In this work, for simplicity and w.l.o.g., we assume such a basis is known.
\begin{assumption}\label{LOassumption: A basis}
    Matrix $A$ is of the form $A = \begin{bmatrix}
        A_B& A_N
    \end{bmatrix}$ with $A_B\in \mathbb{R}^{m\times m}$ being nonsingular and the condition number of $A_B$ being no larger than that of $A$.
\end{assumption}
In fact, Assumption \ref{LOassumption: A basis} is not restrictive because we can rewrite problem \eqref{LOprimal} into the canonical form
\begin{equation*}
    \begin{aligned}
        \min_{x\in\mathbb{R}^n} c^Tx&\\
        {\rm s.t.} \begin{bmatrix}
            A\\-A
        \end{bmatrix}\begin{bmatrix}
            x\\x
        \end{bmatrix}&\leq\begin{bmatrix}
            b\\-b
        \end{bmatrix}\\
        x&\geq 0.
    \end{aligned}
\end{equation*}
Then, we can apply the self-dual embedding model to the canonical form and obtain another standard form LO problem with both Assumption \ref{LOassumption: interior} and \ref{LOassumption: A basis} satisfied. See \cite{mohammadisiahroudi2021inexact} for details.

With Assumption \ref{LOassumption: A basis}, we partition $D$ correspondingly as
\begin{equation*}
    D = \begin{bmatrix}
        D_B&0\\0&D_N
    \end{bmatrix},
\end{equation*}
where $D_B$ is an $m\times m$ matrix and $D_N$ is an $(n-m)\times (n-m)$ diagonal matrix.
We also define $\hat{A} = A_B^{-1}A$ and $\hat{b} = A_B^{-1}b$. In the general case where $A_B$ is not trivial, this one-time computation contributes $\mathcal{O}(mn^2)$ arithmetic operations to total complexity, which can also be avoided if one uses the aforementioned canonical form reformulation and applies the self-dual embedding model as described in \cite{mohammadisiahroudi2021inexact}. Following \cite{mohammadisiahroudi2022efficient}, we can rewrite \eqref{LOsys: normal equation system} as the following modified normal equation system (MNES)
\begin{equation}\label{LOsys: mnes}\tag{MNES}
    \begin{aligned}
        \left(D_B^{\frac{1}{2}}\hat{A}\right)D^{-1}\left(D_B^{\frac{1}{2}}\hat{A}\right)^T z = \left(D_B^{\frac{1}{2}}A_B^{-1}\right)\left(AD^{-1}r_{dc} + r_p\right).
    \end{aligned}
\end{equation}
We introduce the following notations:
\begin{equation*}
    \begin{aligned}
        M_{\rm NES}  &= AD^{-1}A^T\\
        M_{B} &= \left(D_B^{\frac{1}{2}}A_B^{-1}\right)^T\\
        M_{\rm MNES} &= M_B^T M_{\rm NES} M_B\\
        v_{\rm MNES} &=M_B^T\left(AD^{-1}r_{dc} + r_p\right).
    \end{aligned}
\end{equation*}
$\Delta y$ and $z$ satisfies the bijective relationship:
$\Delta y = M_{B} z$.
In  \cite{mohammadisiahroudi2022efficient}, the authors introduced the following procedure to obtain inexact Newton directions.
\begin{enumerate}
    \item[] \textbf{Step 1.} Find $z$ such that $\|r\|_2\leq\eta \sqrt{\mu/{n}}$, where $\eta\in(0,1)$ and 
    \begin{equation}\label{LOeq: r}
        \begin{aligned}
            r &= M_{\rm MNES} z - v_{\rm MNES}.
        \end{aligned}
    \end{equation}
    \item[] \textbf{Step 2.} Calculate \begin{equation}\label{LOeq: delta y and z}
        \Delta y = M_B z.
    \end{equation}
    \item[] \textbf{Step 3.} Calculate $\nu = \begin{bmatrix}
        D_B^{-\frac{1}{2}}r\\ 0
    \end{bmatrix}\in \mathbb{R}^n$.
    \item[] \textbf{Step 4.} Calculate $\Delta s = r_d - A^T  \Delta y$ and $\Delta x = S^{-1}\left(r_c - X\Delta s\right) - \nu$.
\end{enumerate}
We are not using this procedure to calculate Newton steps, but introducing this procedure helps us determine the convergence condition of our proposed algorithm.
For the Newton directions obtained in this way, the residual of ~\eqref{LOsys: Newton system} only shows up in its third equation as $X\Delta s +  S\Delta x - r_c = -S\nu.$ To get the II-QIPM converge, the authors in \cite{mohammadisiahroudi2022efficient} proposed the stopping criteria
\begin{equation*}
    \|S\nu\|_{\infty} \leq \eta \mu.
\end{equation*}
They proved the following lemma.
\begin{lemma}[Lemma 4.2 in \cite{mohammadisiahroudi2022efficient}]\label{LOlemma: convergence condition}
    If $\|r\|_2\leq\eta {\sqrt{\mu/n}}$, then $\|S\nu\|_{\infty} \leq \eta \mu$.
\end{lemma}
To solve \eqref{LOsys: mnes} to the desired accuracy above, i.e., $\|r\|_2\leq \eta \sqrt{\mu/n}$,  the authors in \cite{mohammadisiahroudi2022efficient} analyzed the accuracy required from the quantum algorithm. Specifically, they used the QLSA by \cite{chakraborty2018power} and the QTA by \cite{van2023quantum}.
Before introducing their conclusion on the required accuracy for quantum algorithm, {we} first compare the different meanings of accuracy between classical linear system {solves} and quantum linear system {solves}.
Different from {classical solvers}, QLSAs compute a quantum state representing the normalized vector of the classical solution and store the quantum state in the quantum machine.
We use QTAs to get a classical normalized vector close to the quantum state and then rescale it properly. {Therefore, we take into account the precision of QTAs in the analysis.}

Now we discuss the accuracy needed for the quantum algorithms. Let us consider \eqref{LOsys: normal equation system}. 
When we use classical algorithms to compute an $\epsilon_{\rm c}$-approximate solution $\Delta y$, we get $\Delta\hat{y}$ such that 
\begin{equation*}
   \epsilon_c =  \left\| M_{\rm NES} \Delta \hat{y} - AD^{-1}r_{dc} - r_p\right\|_2 = \left\|M_{\rm NES} (\Delta \hat{y}-\Delta y)\right\|_2. 
\end{equation*}
According to the definition of \eqref{LOsys: mnes} and the definition of $r$, see equation \eqref{LOeq: r}, it is obvious that
\begin{equation}\label{LOeq: error residual}
    \begin{aligned}
        \epsilon_{\rm c} = \left\|M_{\rm NES} (\Delta \hat{y}-\Delta y)\right\|_2 = \left\|\left(M_B^T\right)^{-1} r\right\|_2.
    \end{aligned}
\end{equation}
In contrast, when we use QLSAs and QTAs to solve \eqref{LOsys: normal equation system} with accuracy $\epsilon_{\rm q}$, we get a classical unit vector $\Delta \tilde{y}$ with accuracy $\epsilon_{\rm q}$ such that
\begin{equation}\label{LOexample: quantum error}
    \mathcal{E}_q = \frac{\Delta \tilde{y}}{\|\Delta \tilde{y}\|_2} - \frac{\Delta y}{\|\Delta y\|_2} = \Delta \tilde{y} - \frac{\Delta y}{\|\Delta y\|_2} \quad \text{\rm with }\left\|\mathcal{E}_q\right\|_2 = \epsilon_{\rm q}.
\end{equation}
The norm of the actual solution $\Delta y$ can be retrieved inexactly using quantum algorithms \cite{chakraborty2018power}. But this approach introduces extra inexactness to the algorithm. Instead, we consider rescaling $\Delta \Tilde{y}$ to minimize the $\ell_2$ norm of the residual of \eqref{LOsys: mnes}.
Denote the rescaling factor by $l_{\Delta \Tilde{y}}$. 
The minimization problem is a univariate convex quadratic optimization problem with unique optimal solution
\begin{equation*}
    \begin{aligned}
        l_{\Delta \Tilde{y}}^\ast = \arg\min_{l_{\Delta \Tilde{y}} \in \mathbb{R}} \ \left\|    M_{\rm MNES} M_B^{-1}  l_{\Delta \Tilde{y}} \Delta\Tilde{y} - v_{\rm MNES}\right\|_2.
    \end{aligned}
\end{equation*}
With this rescaling choice, one can show the resulting residual $r_{\rm MNES}$ satisfies
\begin{equation*}
    \begin{aligned}
        \|r_{\rm MNES}\|_2 &\leq \left\|  M_{\rm MNES} M_B^{-1}  \|\Delta y\|_2 \Delta\Tilde{y} - v_{\rm MNES} \right\|_2\\
                         &= \left\|  M_{\rm MNES} M_B^{-1}  (\Delta y + \|\Delta y\|_2 \mathcal{E}_q) - v_{\rm MNES} \right\|_2\\
                         &= \left\|  M_{\rm MNES} M_B^{-1}  \|\Delta y\|_2 \mathcal{E}_q \right\|_2\\
                         &= \left\|  M_{\rm MNES} M_B^{-1}  \mathcal{E}_q \right\|_2 \|\Delta y\|_2\\
                         &= \left\|  M_{\rm MNES} M_B^{-1}  \mathcal{E}_q \right\|_2 \left\| \left( M_{\rm MNES} M_B^{-1}\right)^{-1} v_{\rm MNES}\right\|_2\\
                         &\leq \kappa_{M_{\rm MNES} M_{B}^{-1}} \|v_{\rm MNES}\|_2 \epsilon_{q}.
    \end{aligned}
\end{equation*}
Combining the inequality and Lemma \ref{LOlemma: convergence condition}, we get the following choice of accuracy
\begin{equation}\label{LOeq: convergence condition pre}
    \begin{aligned}
        \epsilon_{\rm QLSA} = \epsilon_{\rm QTA} = \epsilon_q = \frac{\eta}{\kappa_{M_{\rm MNES} M_{B}^{-1}} \|v_{\rm MNES}\|_2} \sqrt{\frac{\mu}{n}}.
    \end{aligned}
\end{equation}
After obtaining the solution from QLSA and QTA, one updates the solution by taking a step along the Newton direction with the step length introduced in \cite{mohammadisiahroudi2022efficient}. 
Before we provide the pseudocode for the II-QIPM proposed in \cite{mohammadisiahroudi2022efficient}, we define the following quantities:
\begin{equation}\label{LOdef:omega}
    \begin{aligned}
        \omega^\ast &= \max_{\mathcal{PD}^\ast}\|(x,\ 0\times y,\ s)\|_{\infty},\\
        \omega^k    &= \|(x^k,\ 0\times y^k,\  s^k)\|_{\infty},\\
        \bar{\omega}&= \max_{k=1,\dots,\bar{k}} \omega^k,
    \end{aligned}
\end{equation}
where $(x^k,\ y^k,\ s^k)$ is the iterate in the $k$th iteration and $\bar{k}$ is the {total number of iterations}. In some sections later, the superscript for iteration index might be ignored. In those context, we simply use $\omega$ for $\omega^k$, i.e., $\omega = \omega^k$.
The pseudocode for the II-QIPM proposed in~\cite{mohammadisiahroudi2022efficient} is provided in Algorithm \ref{LOalg: II-QIPM jota}.

\begin{algorithm}
\caption{II-QIPM~\cite{mohammadisiahroudi2022efficient}}
\label{LOalg: II-QIPM jota}
\begin{algorithmic}[1]
\State{Choose $\epsilon >0$, $\gamma_1\in(0,1)$, $\gamma_2>0$, $0<\eta<\beta_1<\beta_2<1$,}
\State{$k \gets 0$, $(x^0, y^0, s^0) \gets (\omega^{\ast} e, 0 e, \omega^{\ast} e) $, and $\gamma_2 \gets \max\left\{1, \frac{\|(r_p^0,r_d^0)\|_2}{\mu^0}\right\}$,}
\While{$(x^k,y^k,s^k)\notin \mathcal{PD}_\epsilon$}
\State{$\mu^k \gets \frac{(x^k)^Ts^k}{n}$} 
\State{set $\epsilon_{\rm QLSA}^k \text{ and }\epsilon_{\rm QTA}^k$ using \eqref{LOeq: convergence condition pre},} 
\State{$(\Delta x^k,\Delta y^k,\Delta s^k) \gets$ \textbf{solve} \eqref{LOsys: mnes}($\beta_1$) by QLSA+QTA \\ 
 \hspace{4truecm} with precisions $\epsilon_{\rm QLSA}^k \text{ and }\epsilon_{\rm QTA}^k$, respectively,} 
\State{$\hat{\alpha}^k\gets 
    \max \Big\{\bar{\alpha} \in [0,1] \ | \  \text{for all } \alpha \in [0,\bar{\alpha}]$\text{ we have }\\ \hspace{3.5truecm}$\big((x^k,y^k,s^k)+\alpha(\Delta x^k,\Delta y^k,\Delta s^k)\big) \in \mathcal{N}(\gamma_1, \gamma_2)$ \text{ and } \\
    \hspace{3.5truecm}$(x^k+\alpha \Delta x^k)^T(s^k+\alpha \Delta s^k)\leq \big(1-\alpha (1-\beta_2)\big)(x^k)^Ts^k \Big\}$, } 
\State{$(x^{k+1},y^{k+1},s^{k+1}) \gets (x^k,y^k,s^k)+\hat{\alpha}^k(\Delta x^k,\Delta y^k,\Delta s^k)$,} 
\If{$\|x^{k+1},s^{k+1}\|_{\infty}>\omega^\ast$} 
%
\State \textbf{return\ }{Primal or dual is infeasible.}
\EndIf
\State{$k \gets k+1$}
\EndWhile
\State \textbf{return\ }{$(x^k,y^k,s^k)$}
\end{algorithmic}
\end{algorithm}

While II-QIPMs have better dependence on the dimension of the problem, they have linear dependence on the condition number of the coefficient matrix of \eqref{LOsys: mnes}, which in the worst case can be larger than the condition number of the coefficient matrix of \eqref{LOsys: normal equation system}.
It is known that the condition number of normal equation systems can go to infinity for LO problems \cite{guler1993degeneracy}. In this work, we are interested in investigating how to precondition the linear systems arising from II-QIPMs and analyzing how preconditioning can improve the complexity of II-QIPMs.

\section{Preconditioned Quantum Interior Point Method}\label{LOsection: PQIPM}
In this section, we propose a new QIPM using the preconditioning method proposed by Chai and Toh in \cite{chai2007preconditioning}, analyze the condition number of the preconditioned linear system, and prove the complexity of our preconditioned II-QIPM. The preconditioning method we present here is a special adaption of the method proposed in \cite{chai2007preconditioning}. Then, we discuss how to use this method in the QIPM setting and analyze the effect of the proposed preconditioning. Finally, we discuss the complexity of the proposed preconditioned QIPM.

\subsection{A Special Adaptation \& Analysis of Chai \& Toh's Method}
%
As explained in \cite{chai2007preconditioning}, when IPMs approach optimality with duality gap $\mu$, according to the optimal partition of the problem, the diagonal entries in $D$ will go into two clusters, one is of magnitude $\mathcal{O}(\mu)$ and the other $\mathcal{O}(1/\mu)$.
Before the optimal partition is revealed, one can predict the optimal partition according to some partition rules and permute the matrix $D$ into the following form:
\begin{equation*}
    \begin{aligned}
        D=\begin{bmatrix}
        D_1&\\
        &D_2
        \end{bmatrix},
    \end{aligned}
\end{equation*}
where $D_1$ and $D_2$ are both diagonal matrices and the diagonal elements of $D_1$ and $D_2$ are $\mathcal{O}(\mu)$ and $\mathcal{O}({1}/{\mu})$ respectively. Subsequently, one can partition all the related objects in the same manner, i.e.,
\begin{equation*}
    \begin{aligned}
        A = \begin{bmatrix}
        A_1 & A_2
        \end{bmatrix}
    \end{aligned}
\end{equation*}
and
\begin{equation*}
    \begin{aligned}
        \Delta x = \begin{bmatrix}
        \Delta x_1\\
        \Delta x_2
        \end{bmatrix}, \quad r_{dc} = \begin{bmatrix}
        (r_{dc})_1\\(r_{dc})_2
        \end{bmatrix},
    \end{aligned}
\end{equation*}
etc. 
These matrices and vectors are redefined\footnote{Ultimately, when the optimal partition is correctly identified, $\Delta x_1$ and $\Delta x_2$ correspond to the index partition in the optimal partition.} at each iteration, but for simplicity we have dropped the iteration index superscripts here and in the remaining of this section.

With the aforementioned partition, it is shown in \cite{chai2007preconditioning} that ~\eqref{LOsys: augmented system} can be transformed into another linear system as specified in the following lemma.
\begin{lemma}[Lemma 4.1 in \cite{chai2007preconditioning}\footnote{Lemma \ref{LOlemma: rae} is a special case of Lemma 4.1 in \cite{chai2007preconditioning} when their matrix $E_1$ is the identity matrix.}]\label{LOlemma: rae}
The solution of ~\eqref{LOsys: augmented system} can be computed from the following reduced augmented system (RAS):
\begin{equation}\label{LOsys: RAS}\tag{RAS}
    \begin{aligned}
        \begin{bmatrix}
        H & B\\ B^T & -D_1
        \end{bmatrix} \begin{bmatrix}
        \Delta y\\ \Delta \Tilde{x}_1
        \end{bmatrix}=\begin{bmatrix}
        h\\F_1^{-1/2}(r_{dc})_1
        \end{bmatrix}
    \end{aligned},
\end{equation}
where $F_1 = I + D_1$, and
\begin{equation*}
    \begin{aligned}
         \Delta \Tilde{x}_1  &= F_1^{-1/2} \Delta x_1,\\
         B &= A_1 F_1^{-1/2},\\
         H &= A \, {\rm diag}(F_1^{-1}, D_2^{-1}) A^T = BB^T + A_2 D_2^{-1} A_2^T,\\
         h &= r_p + A \, {\rm diag}(F_1^{-1}, D_2^{-1})r_{dc}.
    \end{aligned}
\end{equation*}
\end{lemma}
In this work, we take the following partition rule:
\begin{definition}\label{LOdef: rule}
If $x_{[i]} \geq \sqrt{[(1-\gamma_1)n + \gamma_1]\mu}$, then $i$ is  included in the index set defining partition indicated by subscript $"1"$; else, $i$ is included in the partition indicated by $"2"$.
\end{definition}
\begin{lemma}\label{LOlemma: bound: Psi F D}
Under the partition rule in Definition \ref{LOdef: rule}, the value of $D_{1,[i]}$ lies in the interval $[\frac{\gamma_1 \mu}{\omega^2}, 1]$; the value of $F_{1,[i]}$ lies in the interval $[\frac{\gamma_1\mu}{\omega^2}+1, 2]$; the value of $D_{2,[j]}$ lies in the interval $[\frac{\gamma_1}{(1-\gamma_1)n + \gamma_1}, \frac{\omega^2}{\gamma_1\mu}]$.
\end{lemma}
\begin{proof}
From the definition of the neighbourhood~\eqref{LOdef: neighbourhood}, we have
\begin{align*}
    x_{[i]}s_{[i]} \geq \gamma_1 \mu, \text{ for all } i \in \{1,\dots, n\},
\end{align*}
so
\begin{align*}
    x_{[i]}s_{[i]} = x^Ts - \sum_{j\neq i}x_{[j]}s_{[j]} \leq n\mu - \gamma_1(n-1)\mu.
\end{align*}
From the first inequality we have
\begin{align*}
    D_{1,[i]} = x_{1,[i]}^{-1}s_{1,[i]} = x_{1,[i]}s_{1,[i]}/x_{1,[i]}^2 \geq \frac{\gamma_1 \mu}{x_{1,[i]}^2} \geq \frac{\gamma_1 \mu}{\omega^2},
\end{align*}
and from the second inequality and Definition \ref{LOdef: rule} we derive
\begin{align*}
    D_{1,[i]} = x_{1,[i]}^{-1}s_{1,[i]} = x_{1,[i]}s_{1,[i]}/x_{1,[i]}^2 \leq \frac{n\mu - \gamma_1 (n-1)\mu}{x_{1,[i]}^2} \leq 1.
\end{align*}
Similarly, we can obtain the conclusion for $F_{1,[i]}$ since $F_{1,[i]} = D_{1,[i]} +1$. As for $D_{2,[j]}$, from the definition of the neighbourhood~$\mathcal{N}(\gamma_1,\gamma_2)$, see~\eqref{LOdef: neighbourhood}, we have
\begin{align*}
    D_{2,[j]} = x_{2,[j]}^{-1} s_{2,[j]} = (x_{2,[j]} s_{2,[j]})^{-1} s_{2,[j]}^2 \leq \frac{s_{2,[j]}^2}{\gamma_1\mu} \leq \frac{\omega^2}{\gamma_1\mu}
\end{align*}
and
\begin{align*}
    D_{2,[j]} = x_{2,[j]}^{-1} s_{2,[j]} = \frac{x_{2,[j]} s_{2,[j]}}{x_{2,[j]}^2} \geq \frac{\gamma_1\mu}{x_{2,[j]}^2} \geq \frac{\gamma_1}{(1-\gamma_1)n + \gamma_1},
\end{align*}
that completes the proof.
\end{proof}

Denote the coefficient matrix of \eqref{LOsys: RAS} by $K$, i.e.,
\begin{equation*}
    K = \begin{bmatrix}
        H & B\\ B^T & -D_1
        \end{bmatrix}.
\end{equation*}
It is proven in \cite{chai2007preconditioning} that $K, H$ and $D_1$ are all nonsingular. So one can use the Schur complement to write out the inverse matrix of $K$ as
\begin{equation*}
    K^{-1} = \begin{bmatrix}
    H^{-1} - H^{-1}BZ^{-1}B^TH^{-1} & H^{-1}BZ^{-1}\\
    Z^{-1}B^T H^{-1}                & -Z^{-1}
    \end{bmatrix},
\end{equation*}
where $Z = -(K/H) = B^T H^{-1} B + D_1$.
This naturally leads to the preconditioner with the following block structure
\begin{equation*}
    P_c^{-1} = \begin{bmatrix}
    \hat{H}^{-1} - \hat{H}^{-1}B\hat{Z}^{-1}B^T\hat{H}^{-1} & \hat{H}^{-1}B\hat{Z}^{-1}\\
    \hat{Z}^{-1}B^T \hat{H}^{-1}                & -\hat{Z}^{-1}
    \end{bmatrix},
\end{equation*}
where $\hat{H}$ is a selected positive definite matrix and $\hat{Z}$ is an approximation of $Z$ by using $\hat{H}$ as the approximation of $H$. The preconditioned \eqref{LOsys: RAS} system, that we refer to as P-RAS, is
\begin{equation}\label{LOsys: P_RAS}\tag{P-RAS}
    \begin{aligned}
        P_c^{-1}K = P_c^{-1} \begin{bmatrix}
        h\\F_1^{-1/2}(r_{dc})_1
        \end{bmatrix}.
    \end{aligned}
\end{equation}
The coefficient matrix of \eqref{LOsys: P_RAS} is
\begin{equation*}
    \begin{aligned}
        P_c^{-1}K &= \begin{bmatrix}
        \hat{H}^{-1}H - \hat{H}^{-1}B\hat{Z}^{-1}B^T(\hat{H}^{-1}H - I) & \hat{H}^{-1}B\hat{Z}^{-1}(\hat{Z} - B^T \hat{H}^{-1}B - D_1)\\
        \hat{Z}^{-1}B^T(\hat{H}^{-1}H - I) & \hat{Z}^{-1}(B^T \hat{H}^{-1}B + D_1)
        \end{bmatrix},
    \end{aligned}
\end{equation*}
and the right-hand-side vector is
\begin{equation*}
    \begin{aligned}
        \begin{bmatrix}
        \xi^\prime\\\xi^{\prime\prime}
        \end{bmatrix}= P_c^{-1}\begin{bmatrix}
        h\\F_1^{-\frac{1}{2}}(r_{dc})_1
        \end{bmatrix}=\begin{bmatrix}
        \hat{H}^{-1}h - \hat{H}^{-1}B\hat{Z}^{-1}B^T\hat{H}^{-1}h + \hat{H}^{-1}B\hat{Z}^{-1}F_1^{-\frac{1}{2}}(r_{dc})_1\\
        \hat{Z}^{-1}B^T \hat{H}^{-1}h -\hat{Z}^{-1}F_1^{-\frac{1}{2}}(r_{dc})_1
        \end{bmatrix}.
    \end{aligned}
\end{equation*}
In this work, we choose 
\begin{equation*}
    \begin{aligned}
        \hat{H} &= \hat{h}I\\
        \hat{h} &= \gamma_1\sigma_0^2(A)/\omega^2\\
        \hat{Z} &= B^T \hat{H}^{-1}B + D_1.
    \end{aligned}
\end{equation*} 
Then the coefficient matrix of \eqref{LOsys: P_RAS} can be simplified to
\begin{equation*}
    \begin{aligned}
        P_c^{-1}K &= \begin{bmatrix}
        \hat{H}^{-1}H +  \hat{H}^{-1}B\hat{Z}^{-1}B^T - \hat{H}^{-1}B\hat{Z}^{-1}B^T\hat{H}^{-1}H & 0\\
        \hat{Z}^{-1}B^T(\hat{H}^{-1}H - I) & I
        \end{bmatrix}.
    \end{aligned}
\end{equation*}
Take the following notations:
\begin{equation*}
    \begin{aligned}
        (P_c^{-1}K)_{11} &= \hat{H}^{-1}H +  \hat{H}^{-1}B\hat{Z}^{-1}B^T - \hat{H}^{-1}B\hat{Z}^{-1}B^T\hat{H}^{-1}H\\
        (P_c^{-1}K)_{21} &= \hat{Z}^{-1}B^T(\hat{H}^{-1}H - I).
    \end{aligned}
\end{equation*}
Then \eqref{LOsys: P_RAS} can be simplified as
\begin{equation*}
    \begin{aligned}
        (P_c^{-1}K)_{11}\Delta y &= \xi^\prime\\
        (P_c^{-1}K)_{21}\Delta y + \Delta \Tilde{x}_1&= \xi^{\prime\prime}.
    \end{aligned}
\end{equation*}
Notice that we only need the first equation to compute $\Delta y$. So we call the normalized first equation of \eqref{LOsys: P_RAS} as reduced \eqref{LOsys: P_RAS}, that we refer to as \eqref{LOsys: RP-RAS}, i.e.,
\begin{equation}\label{LOsys: RP-RAS}\tag{RP-RAS}
    \Xi\Delta y = \xi,
\end{equation}
where
\begin{equation*}
    \Xi=\frac{(P_c^{-1}K)_{11}}{\|(P_c^{-1}K)_{11}\|_2},\ \xi = \frac{\xi^\prime}{\|(P_c^{-1}K)_{11}\|_2}.
\end{equation*}
We build and solve \eqref{LOsys: RP-RAS} instead of \eqref{LOsys: mnes} using the QLSA in \cite{chakraborty2018power} and the QTA in \cite{van2023quantum}. We discuss how to build and solve \eqref{LOsys: RP-RAS} in Section \ref{LOsec: solve}. 
Before that, we analyze \eqref{LOsys: RP-RAS} and determine the target accuracy for the quantum subroutine in the next section. The pseudocode for preconditioned II-QIPM is provided here.
\begin{algorithm}
\caption{Preconditioned II-QIPM} \label{LOalg: II-QIPM pre}
\begin{algorithmic}[1]
 \State{Choose $\epsilon >0$,\ $\gamma_1\in(0,1)$,\ $\gamma_2>0$,\ $0<\eta<\beta_1<\beta_2<1$,}
 \State{$k \gets 0$,\ $(x^0, y^0, s^0) \gets (\omega^{\ast} e, 0 e, \omega^{\ast} e) $, and $\gamma_2 \gets \max\left\{1, \frac{\|(r_p^0,r_d^0)\|_2}{\mu^0}\right\}$,}
\While{$(x^k,y^k,s^k)\notin \mathcal{PD}_\epsilon$}
\State{$\mu^k \gets \frac{(x^k)^Ts^k}{n}$} 
\State{Partition according to Definition \ref{LOdef: rule}}
\State{set $\epsilon_{\rm QLSA}^k \text{ and }\epsilon_{\rm QTA}^k$ using \eqref{LOeq: convergence condition pre}} 
\State{$(\Delta x^k,\Delta y^k,\Delta s^k) \gets$ \textbf{build and solve} \eqref{LOsys: RP-RAS} by QLSA+QTA \\
  \hspace{4cm} with precision $\epsilon_{\rm QLSA}^k \text{ and }\epsilon_{\rm QTA}^k$} 
\State{$\begin{aligned}
    \hat{\alpha}^k\gets \max \Big\{\bar{\alpha} \in [0,1] \ | \ & \text{for all } \alpha \in [0,\bar{\alpha}]\text{ we have }\\ &\big((x^k,y^k,s^k)+\alpha(\Delta x^k,\Delta y^k,\Delta s^k)\big) \in \mathcal{N}(\gamma_1, \gamma_2) \text{ and } \\
    &(x^k+\alpha \Delta x^k)^T(s^k+\alpha \Delta s^k)\leq \big(1-\alpha (1-\beta_2)\big)(x^k)^Ts^k \Big\}
\end{aligned}$} 
\State{$(x^{k+1},y^{k+1},s^{k+1}) \gets (x^k,y^k,s^k)+\hat{\alpha}^k(\Delta x^k,\Delta y^k,\Delta s^k)$} 
\If{$\|x_{k+1},s_{k+1}\|_{\infty}>\omega^\ast$} 
\State \textbf{return\ }{Primal or dual is infeasible.}
\EndIf
\State{$k \gets k+1$}
\EndWhile
\State \textbf{return\ }{$(x^k,y^k,s^k)$}
\end{algorithmic}
\end{algorithm}

\subsection{Condition Number Analysis}
In this section, we analyze the condition number of the coefficient matrix of \eqref{LOsys: RP-RAS}, or more precisely, the condition numbers of matrix $\hat{Z}$ and $\Xi$. 
To do so, we define
\begin{equation*}
    \begin{aligned}
            G=\hat{H}^{-1/2}B\hat{Z}^{-1}B^T\hat{H}^{-1/2}
    \end{aligned}
\end{equation*}
and
\begin{equation*}
    \begin{aligned}
        Y = G+ (I-G)\hat{H}^{-1/2}H\hat{H}^{-1/2}.
    \end{aligned}
\end{equation*}
By construction, $Y = (P_c^{-1}K)_{11}$. We first show that the smallest nonzero singular value of matrix $G$ is bounded from below by a constant.
\begin{lemma}\label{LOlemma: Glowerbound}
The smallest nonzero singular value of $G$ satisfies $\sigma_0(G) \geq \frac{1}{1+2\hat{h}\sigma_0^{-2}(A_1)}$.
\end{lemma}
\begin{proof}
Let $J = B^T \hat{H}^{-\frac{1}{2}}$ and denote the singular value decomposition (SVD) of $J$ by $J = Q\Sigma P^T = Q_0 \Sigma_0 P_0^T$, where the first SVD is the full SVD and the second one is the reduced SVD. Then matrix $G$ can be expressed using $J$ and $D_1$,
\begin{equation*}
    \begin{aligned}
    G = J^T(JJ^T + D_1)^{-1}J = P\Sigma^T (\Sigma\Sigma^T +  Q^T D_1 Q)^{-1}\Sigma P^T.
\end{aligned}
\end{equation*}
Let $G^\prime = \Sigma\Sigma^T +  Q^T D_1 Q$. Then $G^\prime$ and $G^{\prime {-1}}$ can be expressed as
\begin{align}
    G^\prime = \begin{bmatrix}
    \Sigma_0^2 + \Phi_{11} & \Phi_{12}\\ \Phi_{21} & \Phi_{22}
    \end{bmatrix} \text{ and }
    G^{\prime {-1}} = \begin{bmatrix}
    (G^{\prime {-1}})_{11} & (G^{\prime{-1}})_{12}\\
    (G^{\prime {-1}})_{21} & (G^{\prime {-1}})_{22}
    \end{bmatrix},
\end{align}
where $\Phi_{ij}$ are submatrices of the matrix $\Phi = Q^T D_1 Q$ and $(G^{\prime -1})_{11} = (\Sigma_0^2 + (\Phi/\Phi_{22}))^{-1}$. 
So matrix $G$ can be expressed as
\begin{align}\label{LOmatrixG2}
    G&=P\begin{bmatrix}
    (I + \Sigma_0^{-1}(\Phi/\Phi_{22})\Sigma_0^{-1})^{-1} &0\\0&0
    \end{bmatrix}P^T
\end{align}
and the singular values of matrix $G$ are the same as the singular values of the matrix in the middle of the right-hand-side expression.
By Lemma \ref{lemma: nonzero singular value of block matrix with zero wrap}, the nonzero singular values of the matrix in the middle equals to the singular values of $\left(I + \Sigma_0^{-1}\left(\Phi/\Phi_{22}\right)\Sigma_0^{-1}\right)^{-1}$, i.e., the reciprocal of the result of the singular values of $\Sigma_0^{-1}\left(\Phi/\Phi_{22}\right)\Sigma_0^{-1}$ plus $1$. By Lemma \ref{lemma: interlace block matrix}, the singular values of $\Sigma_0^{-1}(\Phi/\Phi_{22})\Sigma_0^{-1}$
are bounded from above by the largest singular value of
$\Sigma_0^{-1}\Phi\Sigma_0^{-1}$. Then, by using Lemma \ref{lemma: largest singular value of matrix product}, we have
\begin{align*}
    \sigma_1(\Sigma_0^{-1}\Phi\Sigma_0^{-1})&\leq \sigma_1(\Sigma_0^{-1}) \sigma_1(\Phi)\sigma_1(\Sigma_0^{-1})\\
    &=[\sigma_0(\Sigma_0)]^{-2} \sigma_1(D_1)\\
    &=[\sigma_0(F_1^{-\frac{1}{2}}A_1^T)]^{-2} \hat{h} \sigma_1(D_1)\\
    &\leq [\sigma_0(F_1^{-\frac{1}{2}})\sigma_0(A_1^T)]^{-2}\hat{h}\sigma_1(D_1)\\
    &=\hat{h}\sigma_0^{-2}(A_1)\sigma_1(F_1)\sigma_1(D_1)\\
    &\leq 2\hat{h}\sigma_0^{-2}(A_1),
\end{align*}
where the second inequality follows from Lemma \ref{lemma: least nonzero singular value of matrix product} and the last inequality follows from Lemma \ref{LOlemma: bound: Psi F D}. Thus
\begin{align}
    \sigma_0(G) \geq \frac{1}{1+2\hat{h}\sigma_0^{-2}(A_1)}.
\end{align}
The proof is complete.
\end{proof}

\begin{lemma}\label{LOlemmaGupperbound}
The largest singular value of $G$ satisfies $\sigma_1(G)\leq \left(1+\frac{\hat{h}\gamma_1\mu}{\omega^2}\sigma_1^{-2}(A_1)\right)^{-1}$.
\end{lemma}
\begin{proof}
Similar to the proof for Lemma \ref{LOlemma: Glowerbound}, we have $\sigma_1(G)\leq \left(1+\sigma_0(\Sigma_0^{-1}\Phi\Sigma_0^{-1})\right)^{-1}$. Apply Lemma \ref{lemma: least nonzero singular value of matrix product} twice, then it follows that
\begin{align*}
    \sigma_0(\Sigma_0^{-1}\Phi\Sigma_0^{-1})&\geq \sigma_0(\Sigma_0^{-1}) \sigma_0(\Phi)\sigma_0(\Sigma_0^{-1})\\
    &=[\sigma_1(\Sigma_0)]^{-2} \sigma_0(D_1)\\
    &=[\sigma_1(F_1^{-\frac{1}{2}}A_1^T)]^{-2} \hat{h} \sigma_0(D_1)\\
    &\geq [\sigma_1(F_1^{-\frac{1}{2}})\sigma_1(A_1^T)]^{-2}\hat{h}\sigma_0(D_1)\\
    &=\hat{h}\sigma_1^{-2}(A_1)\sigma_0(F_1)\sigma_0(D_1)\\
    &\geq \frac{\hat{h}\gamma_1\mu}{\omega^2}\sigma_1^{-2}(A_1),
\end{align*}
that completes the proof.
\end{proof}

With the spectral properties of matrix $G$, we then analyze matrix $Y$.  For simplicity of the analysis we denote the SVD of matrix $G$ as
$G = R\Psi R^T = R_0\Psi_0 R_0^T$,
where the former SVD is the full SVD and the latter one is the reduced SVD. To study the properties of matrix $Y$, it is worth mentioning that $Y$ can be represented in the following way
\begin{equation}
    \begin{aligned}\label{LOmatrixYsimplified}
    R^T Y R &= R^T G R + R^T (I-G)\frac{1}{\hat{h}}HR\\
    &= R^T G R + R^T (I-G)RR^T\frac{1}{\hat{h}}HR\\
    &= \Psi + \left(I - \Psi\right) \frac{1}{\hat{h}}R^THR\\
    &= \begin{bmatrix}
           \Psi_0 & 0\\0 &0
    \end{bmatrix} + 
    \begin{bmatrix}
           I - \Psi_0 & 0\\0&I
    \end{bmatrix}\mathcal{A}\\
    &= \begin{bmatrix}
    \Psi_0 + \mu(I - \Psi_0)&\\&\mu I
    \end{bmatrix} + \begin{bmatrix}
    I-\Psi_0&\\&I
    \end{bmatrix}(\mathcal{A} - \mu I),
\end{aligned}
\end{equation}
where
\begin{equation*}
    \begin{aligned}
    \mathcal{A} &= \frac{1}{\hat{h}}R^THR.
\end{aligned}
\end{equation*}
We present the spectral properties of matrix $\mathcal{A}$ in the following lemma and then we study the spectral properties of matrix $Y$.
\begin{lemma}\label{LOlemma: Acal}
The smallest nonzero singular value of $\mathcal{A}$ satisfies $\sigma_0(\mathcal{A})\geq \mu$
and the largest singular value of $\mathcal{A}$ satisfies $\sigma_1(\mathcal{A})\leq \frac{(1-\gamma_1)n + \gamma_1}{\gamma_1\hat{h}} \sigma_1^2(A)$.
\end{lemma}
\begin{proof}
Denote $\begin{bmatrix} F_1^{-\frac{1}{2}}&0\\0&D_2^{-\frac{1}{2}}\end{bmatrix}$ by $M_{FD}$. By Lemma \ref{lemma: least nonzero singular value of matrix product}, it follows that
\begin{align*}
    \sigma_0(\mathcal{A}) = \frac{1}{\hat{h}} \sigma_0(AM_{FD})^2   \geq \frac{1}{\hat{h}}\sigma_0(A)^2 \sigma_0(M_{FD})^2 \geq \frac{\gamma_1 \mu}{\hat{h} \omega^2}\sigma_0(A)^2\geq \mu,
\end{align*}
where the second inequality follows from Lemma \ref{LOlemma: bound: Psi F D}. 
Similarly, by Lemma \ref{lemma: largest singular value of matrix product} and Lemma \ref{LOlemma: bound: Psi F D},
\begin{align*}
    \sigma_1(\mathcal{A}) \leq \frac{(1-\gamma_1)n + \gamma_1}{\gamma_1\hat{h}} \sigma_1^2(A),
\end{align*}
that completes the proof.
\end{proof}
Now we are ready to study the condition number of $Y$. 
\begin{lemma}\label{LOlemma: condition number Y}
Using the partition rule in Definition \ref{LOdef: rule}, the singular values of $Y$ satisfy $\sigma_0(Y) = \Omega\left(\mu\right)$
and
\begin{equation*}
    \sigma_1(Y)\leq 1 + \frac{(1-\gamma_1)n + \gamma_1}{\gamma_1\hat{h}} \sigma_1^2(A),
\end{equation*}
thus the condition number of $Y$ satisfies
\begin{equation*}
    \kappa(Y) = \mathcal{O}\left(\frac{1}{\mu} \left(1 + \frac{(1-\gamma_1)n + \gamma_1}{\gamma_1\hat{h}} \sigma_1^2(A)\right) \right).
\end{equation*}
\end{lemma}
\begin{proof}
Notice that
\begin{align*}
    &\quad\begin{bmatrix}
    \Psi_0 + \mu (I - \Psi_0)&0\\0&\mu  I
    \end{bmatrix} + \begin{bmatrix}
    I-\Psi_0&0\\0&I
    \end{bmatrix}(\mathcal{A} - \mu  I)\\
    &=\left(\left(\begin{bmatrix}
    \Psi_0 + \mu (I - \Psi_0)&0\\0&\mu  I
    \end{bmatrix} + \begin{bmatrix}
    I-\Psi_0&0\\0&I
    \end{bmatrix}(\mathcal{A} - \mu  I)\right)\begin{bmatrix}
    I-\Psi_0&0\\0&I
    \end{bmatrix}\right)\begin{bmatrix}
    I-\Psi_0&0\\0&I
    \end{bmatrix}^{-1}.
\end{align*}
The matrix in the outer parentheses is nonsingular because it is the summation of a  positive diagonal matrix and a positive semidefinite matrix. The matrix outside the parentheses is positive diagonal. Applying Lemma \ref{lemma: least nonzero singular value of matrix product}, it follows that
\begin{align*}
    \sigma_0(Y) &\geq \sigma_0\left(\left(\begin{bmatrix}
    \Psi_0 + \mu (I - \Psi_0)&0\\0&\mu  I
    \end{bmatrix} + \begin{bmatrix}
    I-\Psi_0&0\\0&I
    \end{bmatrix}(\mathcal{A} - \mu  I)\right)\begin{bmatrix}
    I-\Psi_0&0\\0&I
    \end{bmatrix}\right)\\
    &\qquad \times\sigma_0\left(\begin{bmatrix}
    I-\Psi_0&0\\0&I
    \end{bmatrix}^{-1}\right)\\
    &\geq \sigma_0\left(\begin{bmatrix}
    \Psi_0 + \mu (I - \Psi_0)&0\\0&\mu  I
    \end{bmatrix}\begin{bmatrix}
    I-\Psi_0&0\\0&I
    \end{bmatrix}\right)\sigma_0\left(\begin{bmatrix}
    I-\Psi_0&0\\0&I
    \end{bmatrix}^{-1}\right),
\end{align*}
where the second inequality follows from the fact that,  the first matrix in the first line is the summation of a positive diagonal matrix and a positive semidefinite matrix, and thus we only keep the positive diagonal one. Then, using Lemma \ref{LOlemmaGupperbound} and Lemma \ref{LOlemma: Glowerbound}, we can get
\begin{equation*}
    \begin{aligned}
        \sigma_0(Y) &= \min\left\{ \sigma_0(G)(1-\sigma_1(G)) + \mu (1-\sigma_1(G))^2,\ \mu \right\} \times 1={\Omega}\left(\mu\right).
    \end{aligned}
\end{equation*}
Recall that \eqref{LOsys: mnes} is equivalent to \eqref{LOsys: normal equation system}, thus $Y$ has no zero singular value, so $\sigma_m(Y) = \sigma_0(Y)$.
As for the largest singular value, we have the following bound:
\begin{align*}
    \sigma_1(Y) \leq 1 + 1\times\left(\sigma_1(\mathcal{A}) - \mu \right) \leq 1 + \frac{(1-\gamma_1)n + \gamma_1}{\gamma_1\hat{h}} \sigma_1^2(A).
\end{align*}
The proof is complete.
\end{proof}
We have proved that the condition number of the coefficient matrix of \eqref{LOsys: mnes} is $\mathcal{O}(1/\mu).$ However, when constructing the coefficient matrix, the inverse matrix of $\hat{Z}$ is also needed. Thus the condition number of $\hat{Z}$ also matters.
\begin{lemma}\label{LOlemma: condition number Zhat}
The singular values of $\hat{Z}$ satisfy $\sigma_0(\hat{Z})\geq {\gamma_1\mu }/{\omega^2}$ and $\sigma_1(\hat{Z})\leq 1+\frac{1}{\hat{h}}\sigma_1^2(A_1)$; the condition number of $\hat{Z}$ satisfies $\kappa(\hat{Z}) \leq \frac{1}{\mu } \frac{\omega^2}{\gamma_1}(1 + \frac{1}{\hat{h}}\sigma_1(A)^2)$.
\end{lemma}
\begin{proof}
By the definition of $\hat{Z}$ that $\hat{Z} = B\hat{H}^{-1}B^T + D_1$ and Lemma \ref{LOlemma: bound: Psi F D}, the smallest nonzero singular value of $\hat{Z}$ satisfies $\sigma_0(\hat{Z}) \geq 0 + \sigma_0(D_1) \geq {\gamma_1\mu }/{\omega^2}$. The largest singular value of $\hat{Z}$ satisfies
\begin{align*}
    \sigma_1(\hat{Z}) &\leq \frac{1}{\hat{h}}\sigma_1(BB^T) + \sigma_1(D_1)\\
                    &\leq \frac{1}{\hat{h}}\sigma_1(A_1)^2\sigma_0(F_1)^{-1} + \sigma_1(D_1)\\
                    &\leq \frac{1}{\hat{h}}\sigma_1(A_1)^2  + 1,
\end{align*}
where the first inequality holds because both terms in $\hat{Z}$ are positive definite, the second inequality follows from Lemma \ref{lemma: largest singular value of matrix product}, and the third follows from Lemma \ref{LOlemma: bound: Psi F D}. Combining the above with $\sigma_1(A_1)\leq \sigma_1(A)$, we have the claimed result.
\end{proof}

In the sequel we use $\kappa_Y$ and $\kappa(Y)$ (resp. $\kappa_{\hat{Z}}$ and $\kappa(\hat{Z})$) interchangeably.  In the next section, we describe how to use QLSAs and QTAs to solve \eqref{LOsys: RP-RAS}.


\subsection{Solve \eqref{LOsys: RP-RAS} w. QLSAs \& QTAs}\label{LOsec: solve}
In this work, we follow the II-QIPM introduced in \cite{mohammadisiahroudi2022efficient} and use the QLSA proposed in \cite{gilyen2018quantum} as well as the QTA proposed in \cite{van2023quantum} to solve \eqref{LOsys: RP-RAS}. 
In this section, we discuss the block-encoding of  \eqref{LOsys: RP-RAS}.
We assume that we have access to quantum RAM (QRAM) and have the initial data stored in QRAM.
{For those operations that contribute polylogarithmic overhead, we say those operations can be implemented \textit{efficiently} without spelling out their complexity.}
We also assume we can efficiently construct diagonal matrices from vectors with the diagonal entries being powers of the entries of vectors. For example, if we have a vector $x$ in QRAM, then we can efficiently construct $X$ and $X^{-1/2}$ in QRAM. Finally, we assume that we can efficiently construct submatrices of any matrix stored in QRAM.

Provided access to QRAM, the complexity associated with block-encoding the involved matrices and preparing a quantum state encoding of the right-hand-side amounts to polylogarithmic overhead, which is dominated by the cost of the block-encoding of negative powers of matrices and that of QLSAs and QTAs.
{Thus, we} ignore the cost of the block-encoding of matrices here and, in turn, we can ignore the accuracy parameter of block-encoding in our analysis -- we keep the accuracy parameters of block-encoding in our analysis but we do not analyze their value.

We start with the block-encoding of $B$. Recall that the definition of $B$ is provided in Lemma \ref{LOlemma: rae} as $B = A_1 F_1^{-1/2}$. We have the following lemma.
\begin{lemma}[Block-encodings of $A$ and $A_1$]\label{LOlemma: BE: A}
    An $\left(\|A\|_F, \log(n) +2, \epsilon_A \right)$-block-encoding of $A$ and an $\left(\|A_1\|_F, \log(n)+2,\epsilon_{A_1}  \right)$-block-encoding of $A_1$ can be implemented efficiently.
\end{lemma}
\begin{proof}
    The results follow directly from the two assumptions we just made earlier in this section and Lemma 50 from \cite{gilyen2018quantum}.
\end{proof}
The following lemma indicates that the block-encoding of a diagonal matrix with entries bounded between $-1$ and $1$ can be implemented easily.
\begin{lemma}[Block encode diagonal matrices]\label{LOlemma: BE: diagonal}
    Let $M\in\mathbb{R}^{n\times n}$ and $\Tilde{M}\in\mathbb{R}^{n\times n}$ be two diagonal matrices and let us assume that for all $i\in\{1,\dots,n\}$, the following conditions hold:
    $$M_{[i]}\in [-1,1],\ \Tilde{M}_{[i]}\in [-1,1],\  M_{[i]}^2 + \Tilde{M}_{[i]}^2 = 1.$$
    Then, a $\left(1, \mathcal{O}(\log(n)), 0 \right)$-block-encoding of $M$ can be implemented efficiently\footnote{Note that a $(p_1, p_2, p_3)$-block-encoding can be implemented efficiently indicates that the cost of the block-encoding is polylogarithmic in terms of $p_1, p_2$, and $p_3$.}.
\end{lemma}
\begin{proof}
    By construction,
    \begin{equation*}
        \begin{bmatrix}
            M & \Tilde{M}\\ \Tilde{M}& -M
        \end{bmatrix}
        \begin{bmatrix}
            M & \Tilde{M}\\ \Tilde{M}& -M
        \end{bmatrix}^T
        =
        \begin{bmatrix}
            M^2 + \Tilde{M}^2 & 0\\ 0& M^2 + \Tilde{M}^2
        \end{bmatrix}= I.
    \end{equation*}
    Then according Lemma 50 from \cite{gilyen2018quantum} the conclusion holds.
\end{proof}
Based on this lemma, we can have the following corollary.
\begin{corollary}\label{LOcoro: BE diagonal}
Let $M\in \mathbb{R}^{n\times n}$ be a diagonal matrix and $\|M\|_{\max}$ be the maximum of the absolute value of the entries of $M$. Then, an $\left(\alpha_M, \mathcal{O}(\log(n)), 0 \right)$-block-encoding of $M$ can be implemented efficiently, where $\alpha_M\geq \|M\|_{\max}$.
\end{corollary}
Then we can efficiently implement the block-encoding of diagonal matrices with proper scaling. Now we discuss the block-encoding of matrix $B$.
\begin{lemma}\label{LOlemma: BE B}
  An \[\left(\|A_1\|_F \|F_1^{-1/2}\|_{\max}, \mathcal{O}(\log(n)), \epsilon_{B} \right) \]-block-encoding of matrix $B$ can be implemented efficiently. 
\end{lemma}
\begin{proof}
    Using Lemma \ref{LOlemma: BE: A}, we can implement an \[\left(\|A_1\|_F, \log(n_1)+2,\epsilon_{A_1}  \right)\]-block-encoding of $A_1$ efficiently.
    From Corollary \ref{LOcoro: BE diagonal}, we can implement an \[\left(\|F_1^{-1/2}\|_{\max}, \mathcal{O}(\log(n)), 0\right)\]-block-encoding of $F_1^{-1/2}$ efficiently. Then, by Lemma 50 of \cite{gilyen2018quantum}, we can implement an \[\left(\|A_1\|_F \|F_1^{-1/2}\|_{\max}, \mathcal{O}(\log(n)), \epsilon_{B} \right) \]-block-encoding of $B$ efficiently.
\end{proof}
Now we are ready to discuss the block-encoding of  matrix $H$ that is defined in Lemma \ref{LOlemma: rae} as $H=A\;{\rm diag}\left(F_1^{-1}, D_2^{-1} \right)A^T$.

\begin{lemma}\label{LOlemma: BE H}
    An \[\left(\|A\|_F^2\max\{\|F_1^{-1}\|_{\max} , \|D_2^{-1}\|_{\max} \}, \mathcal{O}(\log(n)), \epsilon_{H}  \right)\]-block-encoding of matrix $H$ can be implemented efficiently. 
\end{lemma}
\begin{proof}
    From Corollary \ref{LOcoro: BE diagonal}, we can implement a \[\left( \max\{\|F_1^{-1}\|_{\max} , \|D_2^{-1}\|_{\max} \}, \mathcal{O}(\log(n)), 0 \right) \]-block-encoding of ${\rm diag}\left(F_1^{-1}, D_2^{-1} \right)$ efficiently. Following Lemma 53 of \cite{gilyen2018quantum} and Lemma \ref{LOlemma: BE: A}, we can implement a \[\left(\|A\|_F^2\max\{\|F_1^{-1}\|_{\max} , \|D_2^{-1}\|_{\max} \}, \mathcal{O}(\log(n)),  \epsilon_{H} \right)\]-block-encoding of matrix $H$ efficiently.
\end{proof}
Then we discuss the block-encoding of matrix
$\hat{Z} = B^T\hat{H}^{-1}B + D_1 = \frac{1}{\hat{h}} B^TB + D_1$.
\begin{lemma}\label{LOlemma: BE Z hat}
    An  \[\left(\alpha_H\|A_1\|_F^2 \|F_1^{-1/2}\|_{\max}^2, \mathcal{O}(\log(n)), \epsilon_{\hat{Z}} \right)\]-block-encoding of $\hat{Z}$ can be implemented efficiently, where $\alpha_H = \mathcal{O}\left( 1/\hat{h}\right)$.
\end{lemma}
\begin{proof}
    From Lemma 53 of \cite{gilyen2018quantum} and Lemma \ref{LOlemma: BE B}, we can implement an \[\left(\|A_1\|_F^2 \|F_1^{-1/2}\|_{\max}^2, \mathcal{O}(\log(n)), \epsilon_{B^TB} \right) \]-block-encoding of $B^TB$ efficiently.
    From Lemma \ref{LOlemma: bound: Psi F D}, we have $\|D_1\|_{\max}\leq 1$. Then from Corollary \ref{LOcoro: BE diagonal}, we can implement an \[\left(\|A_1\|_F^2 \|F_1^{-1/2}\|_{\max}^2, \mathcal{O}(\log(n)), 0  \right)\]-block-encoding of $\|A_1\|_F^2 \|F_1^{-1/2}\|_{\max}^2D_1$ efficiently. 
    From Lemma 52 of \cite{gilyen2018quantum}, we can implement an \[\left(\alpha_H\|A_1\|_F^2 \|F_1^{-1/2}\|_{\max}^2, \mathcal{O}(\log(n)), \epsilon_{\hat{Z}} \right)\]-block-encoding of $\hat{Z}$ efficiently, where
    \begin{equation*}
        \begin{aligned}
            \alpha_H &=  1/\hat{h}+ 1/\|A_1\|_F^2 \|F_1^{-1/2}\|_{\max}^2 =\mathcal{O}\left( 1/\hat{h}  \right).
        \end{aligned}
    \end{equation*}
    The last equation follows from Lemma \ref{LOlemma: bound: Psi F D}.
\end{proof}
Then we discuss the block-encoding of  $\hat{Z}^{-1}$. 
\begin{lemma}\label{LOlemma: BE Z inv}
    A \[\left(2/\|\hat{Z}\|_2, \mathcal{O}\left({{\rm poly}\log(n)} +{\rm poly}\log\kappa_{\hat{Z}} \right), \epsilon_{\hat{Z}^{-1}} \right)\]-block-encoding of $\hat{Z}^{-1}$ can be implemented with cost
    \[\tilde{\mathcal{O}}_{n, \epsilon}\left(\kappa_{\hat{Z}}\left(\alpha_H\|A_1\|_F^2 \|F_1^{-1/2}\|_{\max}^2\right)/\|\hat{Z}\|_2\right).\]
\end{lemma}
\begin{proof}
    By applying Lemma 52 of \cite{gilyen2018quantum} to $\hat{Z}$, we can efficiently implement an \[\left(\left(\alpha_H\|A_1\|_F^2 \|F_1^{-1/2}\|_{\max}^2\right)/\|\hat{Z}\|_2, \mathcal{O}(\log(n)), \left(\alpha_H\|A_1\|_F^2 \|F_1^{-1/2}\|_{\max}^2\right)/\|\hat{Z}\|_2\epsilon_{\hat{Z}} \right)\]-block-encoding of $\hat{Z}/\|\hat{Z}\|_2$. By applying Lemma 10 of \cite{chakraborty2018power} to $\hat{Z}/\|\hat{Z}\|_2$, a
    \[\left(2, \mathcal{O}\left({{\rm poly}\log(n)} +{\rm poly}\log\kappa_{\hat{Z}} \right), \epsilon_{\|\hat{Z}\|_2\hat{Z}^{-1}}  \right)\]
    -block-encoding of $\|\hat{Z}\|_2\hat{Z}^{-1}$ can be implemented with cost
    \[\tilde{O}_{n, \epsilon}\left(\kappa_{\hat{Z}}\left(\alpha_H\|A_1\|_F^2 \|F_1^{-1/2}\|_{\max}^2\right)/\|\hat{Z}\|_2\right).\]
    Then the result follows from Lemma 52 of \cite{gilyen2018quantum}.
\end{proof}
Then we discuss the block-encoding of  $B\left(\hat{Z}^{-1} \right)B^T$.
\begin{lemma}\label{LOlemma: BE BZinvB}
    A 
    \[\left(2\left(\|A_1\|_F^2 \|F_1^{-1/2}\|_{\max}^2\right)/\|\hat{Z}\|_2, \mathcal{O}\left({{\rm poly}\log(n)} +{\rm poly}\log\kappa_{\hat{Z}} \right), \epsilon_{BZB} \right) \]
    -block-encoding of $B\left(\hat{Z}^{-1} \right)B^T$ can be implemented efficiently.
\end{lemma}
\begin{proof}
    The result follows from Lemma \ref{LOlemma: BE B} and Lemma \ref{LOlemma: BE Z inv} and Lemma 53 of \cite{gilyen2018quantum}.
\end{proof}
Then we discuss the block-encoding $H + B\hat{Z}^{-1}B^T$.
\begin{lemma}\label{LOlemma: BE HBZB}
    A
    \begin{equation*}
        \begin{pmatrix}
            2\left(\|A_1\|_F^2 \|F_1^{-1/2}\|_{\max}^2\right)/\|\hat{Z}\|_2 + \|A\|_F^2\max\{\|F_1^{-1}\|_{\max} ,
            \|D_2^{-1}\|_{\max} \}, \\
            \mathcal{O}\left({\rm poly}\log(n)+{\rm poly}\log\kappa_{\hat{Z}} \right), \\
            \epsilon_{HBZB}
        \end{pmatrix}
    \end{equation*}
    -block-encoding of $H+ B\hat{Z}^{-1}B^T$ can be implemented efficiently.
\end{lemma}
\begin{proof}
    The result follows from Lemma \ref{LOlemma: BE H}, Lemma \ref{LOlemma: BE BZinvB}, and Lemma 52 of \cite{gilyen2018quantum}.
\end{proof}
\begin{lemma}\label{LOlemma: BE BZBH}
    A 
    \begin{equation*}
        \begin{aligned}
            \begin{pmatrix}
                &2\left(\|A_1\|_F^2 \|F_1^{-1/2}\|_{\max}^2\right)\|A\|_F^2\max\{\|F_1^{-1}\|_{\max} ,
            \|D_2^{-1}\|_{\max} \} /\|\hat{Z}\|_2,\\
            &\mathcal{O}\left(\{{\rm poly}\log(n)+{\rm poly}\log\kappa_{\hat{Z}} \right), \\
            &\epsilon_{BZBH}
            \end{pmatrix}
        \end{aligned}
    \end{equation*}
    -block-encoding of $B\hat{Z}^{-1}B^TH$ can be implemented efficiently.
\end{lemma}
\begin{proof}
    The result follows from Lemma \ref{LOlemma: BE H}, Lemma \ref{LOlemma: BE BZinvB}, and Lemma 53 of \cite{gilyen2018quantum}.
\end{proof}
Finally, we have the block-encoding of \eqref{LOsys: RP-RAS}.
\begin{lemma}\label{LOlemma: BE for RPRAS}
Let
\begin{equation*}
    \begin{aligned}
        C_{\rm RPRAS} = \frac{\left(2\hat{h}\frac{\left(\|A_1\|_F^2 \|F_1^{-1/2}\|_{\max}^2\right)}{\left\|\hat{Z} \right\|_2}+1\right)\left(\|A\|_F^2\max\{\|F_1^{-1}\|_{\max} , \|D_2^{-1}\|_{\max} \}+1\right)-1}{\hat{h}^2 \left\| \left( P_c^{-1}K\right)_{11} \right\|_2}.
    \end{aligned}
\end{equation*}
Then a \[\left(C_{\rm RPRAS},
\mathcal{O}\left({{\rm poly}\log(n)} +{\rm poly}\log\kappa_{\hat{Z}} \right),
\epsilon_{\rm RPRAS} \right) \]-block-encoding of $\Xi$ can be implemented {with} $
    \tilde{O}_{n, \epsilon}\left(\kappa_{\hat{Z}}\left(\alpha_H\|A_1\|_F^2 \|F_1^{-1/2}\|_{\max}^2\right)/\|\hat{Z}\|_2\right)$ {queries to QRAM}.
\end{lemma}
\begin{proof}
The result follows from Lemma \ref{LOlemma: BE HBZB}, Lemma \ref{LOlemma: BE BZBH} and Lemma 52 of \cite{gilyen2018quantum}.
\end{proof}
We have the following lemma for the value of $C_{\rm RPRAS}$.
\begin{lemma}\label{LOlemma: bound C}
    $C_{RPRAS} = \mathcal{O}\left(\frac{\|A\|_F^4}{\hat{h} \left\|\hat{Z}\right\|_2  \left\| Y \right\|_2  }  \right).$
\end{lemma}
\begin{proof}
From Lemma \ref{LOlemma: bound: Psi F D}, we have
\begin{equation*}
    \begin{aligned}
        \left\|F_1^{-1/2}\right\|_{\max}^2 = \mathcal{O}(1)
    \end{aligned}
\end{equation*}
and
\begin{equation*}
    \begin{aligned}
        \max\{\|F_1^{-1}\|_{\max} , \|D_2^{-1}\|_{\max} \} = \mathcal{O}(1). 
    \end{aligned}
\end{equation*}
Then
\begin{equation*}
    \begin{aligned}
        C_{RPRAS} &= \mathcal{O}\left(\frac{\|A\|_F^4}{\hat{h} \left\|\hat{Z}\right\|_2  \left\| \left( P_c^{-1}K\right)_{11} \right\|_2  }  \right) = \mathcal{O}\left(\frac{\|A\|_F^4}{\hat{h} \left\|\hat{Z}\right\|_2  \left\| Y \right\|_2  }  \right).
    \end{aligned}
\end{equation*}
The proof is complete.
\end{proof}

\subsection{Worst-case Complexity}

In this section, we analyze the complexity of Algorithm \ref{LOalg: II-QIPM pre}. 
In each iteration of Algorithm \ref{LOalg: II-QIPM pre}, we use Definition \ref{LOdef: rule} to partition the variables. We use the QLSA of \cite{chakraborty2018power} to build and solve \eqref{LOsys: RP-RAS}, use the QTA from \cite{van2023quantum} to readout the solution, and then use line search to update the iterates. The complexity is dominated by the complexity of QLSA and QTA.
\begin{theorem}\label{LOtheorem: main}
In the $k^{\rm th}$ iteration of Algorithm \ref{LOalg: II-QIPM pre}, the QLSA by \cite{chakraborty2018power} and the QTA by \cite{van2023quantum} can build and solve \eqref{LOsys: RP-RAS} with a solution satisfying $\|\hat{r}^k\|_2\leq \eta{\sqrt{\mu^k/{n}}}$ with 
\[\tilde{\mathcal{O}}_{n,\bar{\omega}, \frac{1}{\epsilon}}\left( \frac{ n^{1.5}\bar{\omega}^{13}}{\epsilon^6}\kappa_A^5 \|A\|_F^6 \left( \left\|\hat{A}  \right\|_2 + \left\|\hat{b}  \right\|_2 \right)\right)\]
queries to QRAM.
The total complexity for Algorithm \ref{LOalg: II-QIPM pre} to generate an $\epsilon$-approximate solution is
\[\tilde{\mathcal{O}}_{n, \bar{\omega}, \frac{1}{\epsilon}}\left(n^2\cdot\left( \frac{ n^{1.5}\bar{\omega}^{13}}{\epsilon^6}\kappa_A^5 \|A\|_F^6 \left( \left\|\hat{A}  \right\|_2 + \left\|\hat{b}  \right\|_2 \right)\right)\right)\]
queries to QRAM and $\tilde{\mathcal{O}}_{n,\bar{\omega},\frac{1}{\epsilon}}(n^3m)$ classical arithmetic operations.
\end{theorem}
\begin{proof}
Let $T_{\rm RPRAS}$ be the cost of implementing the block-encoding of $\Xi$ as in Lemma \ref{LOlemma: BE for RPRAS}. Then, we have
\begin{equation*}
    T_{\rm RPRAS} = \tilde{O}_{n, \epsilon}\left(\kappa_{\hat{Z}}\left(\alpha_H\|A_1\|_F^2 \|F_1^{-1/2}\|_{\max}^2\right)/\|\hat{Z}\|_2\right).
\end{equation*}
Let $T_{\xi}$ be the cost of implementing state-preparation of $\xi$ in equation \eqref{LOsys: RP-RAS} as introduced in \cite{chakraborty2018power}. 
Then, the complexity for QLSA + QTA, that we denote by $T_{\rm iter}$, is 
\begin{equation*}
    \begin{aligned}
    T_{\rm iter} = \tilde{\mathcal{O}}_{n, \kappa_{\hat{Z}}, \kappa_{Y}}\left(\kappa_{Y}\left(C_{\rm RPRAS} T_{\rm RPRAS} \log^2\left(\frac{\kappa}{\epsilon_{QLSA}}\right) + T_{\xi}\right)\frac{n}{\epsilon_{QTA}}\right)
\end{aligned}
\end{equation*}
{queries to QRAM.}
According to Corollary 2 of \cite{chakraborty2018power}, $T_{\xi}$ is $\mathcal{O}\left({\rm polylog} n/\epsilon
 \right)$, which is dominated by $T_{\rm RPRAS}$. 
 Thus we ignore $T_{\xi}$ in the complexity of QLSA + QTA.
Following from Lemma \ref{LOlemma: bound C}, and Lemma \ref{LOlemma: BE Z inv}, we have
\begin{equation*}
    \begin{aligned}
    T_{\rm iter}&= \tilde{\mathcal{O}}_{n,\omega, \frac{1}{\epsilon}}\Bigg(\kappa_{Y} \frac{\|A\|_F^4}{\hat{h} \left\|\hat{Z}\right\|_2  \left\| Y \right\|_2  } \kappa_{\hat{Z}}\left(\alpha_H\|A_1\|_F^2 \|F_1^{-1/2}\|_{\max}^2\right)/\|\hat{Z}\|_2 \frac{n}{\epsilon_{QTA}}\Bigg).
    \end{aligned}
\end{equation*}
Then, using Lemma \ref{LOlemma: bound: Psi F D}, we have
\begin{equation*}
    \begin{aligned}
    T_{\rm iter}&=\tilde{\mathcal{O}}_{n,\omega, \frac{1}{\epsilon}}\Bigg(\kappa_{Y} \frac{\|A\|_F^6}{\hat{h} \left\|\hat{Z}\right\|_2  \left\| Y \right\|_2  } \kappa_{\hat{Z}}\alpha_H/\|\hat{Z}\|_2 \frac{n}{\epsilon_{QTA}}\Bigg).
    \end{aligned}
\end{equation*}
Now we can use Lemma \ref{LOlemma: BE Z hat} to derive
\begin{equation*}
    \begin{aligned}
        T_{\rm iter}&=\tilde{\mathcal{O}}_{n,\omega, \frac{1}{\epsilon}}\Bigg(\kappa_{Y} \frac{\|A\|_F^6}{\hat{h}^2 \left\|\hat{Z}\right\|_2  \left\| Y \right\|_2  } \kappa_{\hat{Z}}/\|\hat{Z}\|_2 \frac{n}{\epsilon_{QTA}}\Bigg).
    \end{aligned}
\end{equation*}
Then, by the definition of the condition number, we have
\begin{equation*}
    \begin{aligned}
        T_{\rm iter}&=\tilde{\mathcal{O}}_{n,\omega, \frac{1}{\epsilon}}\left(\frac{\|A\|_F^6\kappa_{\hat{Z}}}{\hat{h}^2 \sigma_1(\hat{Z})^2   \sigma_0(Y)  } \frac{n}{\epsilon_{QTA}}\right).
    \end{aligned}
\end{equation*}
Following Lemma \ref{lemma: interlace}, we have
\begin{equation*}
    \sigma_1(\hat{Z}) \geq \sigma_1(B^T\hat{H}^{-1}B).
\end{equation*}
By the definition of the largest singular value, we have
\begin{equation*}
    \begin{aligned}
        \sigma_1(B^T\hat{H}^{-1}B) = \max_{\|v\|_2=1} v^T B^T\hat{H}^{-1}B v \geq \frac{1}{\hat{h}}\begin{bmatrix}
            1&0&\cdots&0
        \end{bmatrix} B^TB\begin{bmatrix}
            1&0&\cdots&0
        \end{bmatrix}^T.
    \end{aligned}
\end{equation*}
Combining this with Assumption \ref{LOassumption: integer data} and Lemma \ref{LOlemma: bound: Psi F D}, we have
\begin{equation*}
    \sigma_1(\hat{Z}) \geq \frac{1}{2\hat{h}}.
\end{equation*}
Then, using the required precision set by equation~\eqref{LOeq: convergence condition pre}, we have the {total number of queries to QRAM in } the $k^{\rm th}$ iteration of Algorithm \ref{LOalg: II-QIPM pre} equals to 
\begin{equation*}
    \begin{aligned}
        T_{\rm iter}=&\tilde{\mathcal{O}}_{n,\omega, \frac{1}{\epsilon}}\Bigg(\|A\|_F^6\frac{\kappa_{\hat{Z}}}{  \sigma_0(Y)  } \frac{n}{\epsilon_{QTA}}\Bigg) \\
        =&\tilde{\mathcal{O}}_{n,\omega, \frac{1}{\epsilon}}\Bigg(\|A\|_F^6\frac{n^{1.5}\kappa_{\hat{Z}}}{ \sigma_0(Y)  }  \frac{\kappa_{M_{\rm MNES} M_{B}^{-1}} \|v_{\rm MNES}\|_2}{\sqrt{\mu}}\Bigg) \\
        =& \tilde{\mathcal{O}}_{n,\omega, \frac{1}{\epsilon}}\Bigg(\|A\|_F^6\frac{n^{1.5}}{ \mu  } \frac{1}{\mu } \frac{\omega^2}{\gamma_1}(1 + \frac{1}{\hat{h}}\sigma_1(A)^2) \frac{\kappa_{M_{\rm MNES} M_{B}^{-1}} \|v_{\rm MNES}\|_2}{\sqrt{\mu}}\Bigg)\\
        =& \tilde{\mathcal{O}}_{n,\omega, \frac{1}{\epsilon}}\Bigg(\|A\|_F^6\frac{n^{1.5} \omega^2}{ \mu^2  } \frac{1}{\hat{h}}\sigma_1(A)^2 \frac{\kappa_{M_{\rm MNES} M_{B}^{-1}} \|v_{\rm MNES}\|_2}{\sqrt{\mu}}\Bigg),
    \end{aligned}
\end{equation*}
where the second equality follows from Lemma \ref{LOlemma: condition number Zhat}.
By the definition of $\hat{h}$, we have
\begin{equation*}
    \begin{aligned}
        \frac{1}{\hat{h}}\sigma_1(A)^2 = \frac{\kappa_A^2\omega^2}{\gamma_1}.
    \end{aligned}
\end{equation*}
By the definition of the condition number and the sub-multiplicativity of spectral norm, we have
\begin{equation*}
    \begin{aligned}
        \kappa_{M_{\rm MNES} M_{B}^{-1}} &= \kappa_{M_{B}^T M_{\rm NES}}\\
        &\leq \kappa_{M_B} \kappa_{\rm NES}\\
        &\leq \kappa_A^3 \kappa_D^{1.5},
    \end{aligned}
\end{equation*}
where the last inequality holds due to Assumption~\ref{LOassumption: A basis} and the fact that $D$ is a positive diagonal matrix.
It follows from Lemma~\ref{LOlemma: bound: Psi F D} that
\begin{equation*}
    \begin{aligned}
       \kappa_D &\leq  \frac{\omega^4}{\mu^2}.
    \end{aligned}
\end{equation*}
From the proof of part (ii) of Theorem 4.2 in \cite{mohammadisiahroudi2022efficient}, we have
\begin{equation*}
    \begin{aligned}
        \frac{\|v_{\rm MNES}\|_2}{\sqrt{\mu}}=\mathcal{O}\left(\frac{\omega^3 \left( \left\|\hat{A}  \right\|_2 + \left\|\hat{b}  \right\|_2 \right)}{\mu}\right).
    \end{aligned}
\end{equation*}
Putting all this together, {in each IPM iteration we need}
\begin{equation*}
    \begin{aligned}
        \tilde{\mathcal{O}}_{n,\omega, \frac{1}{\epsilon}}\Bigg(\frac{n^{1.5} \omega^{13}}{ \mu^6  }\kappa_A^5  \|A\|_F^6 \left( \left\|\hat{A}  \right\|_2 + \left\|\hat{b}  \right\|_2 \right)\Bigg)
    \end{aligned}
\end{equation*}
{queries to QRAM.}
Recall the relationship between $\omega$, $\omega^k$ and $\bar{\omega}$ in equation \eqref{LOdef:omega}, and the stopping criteria, $\mu\leq \epsilon$. The total { number of queries to QRAM} is
\[\tilde{\mathcal{O}}_{n, \bar{\omega}, \frac{1}{\epsilon}}\left(n^2\cdot\left( \frac{ n^{1.5}\bar{\omega}^{13}}{\epsilon^6}\kappa_A^5 \|A\|_F^6 \left( \left\|\hat{A}  \right\|_2 + \left\|\hat{b}  \right\|_2 \right)\right)\right).\]
{And the total classical arithmetic operations for building \eqref{LOsys: RP-RAS} and updating variables is
\begin{equation*}
    \tilde{\mathcal{O}}_{n, \bar{\omega}, \frac{1}{\epsilon}}\left(n^3m\right),
\end{equation*}
where $\mathcal{O}(nm)$ is the number of classical arithmetic operations needed for matrix-vector multiplication in the updating process.
}
\end{proof}

The {total number of queries to QRAM} of the preconditioned II-QIPM is better than the original II-QIPM proposed in \cite{mohammadisiahroudi2022efficient} with a improvement factor $\mathcal{O}\left( \frac{\bar{\omega}^6}{\epsilon^3} \right)$ {when $\|A\|_F = \mathcal{O}(\|A\|_2)$}. The improvement comes from two sources: one is the preconditioning and the other is different rescaling. 

Preconditioning mainly improves the complexity of the QLSA, which relies on the condition number of the linear system to be solved and the complexity to build the linear system using block-encoding. Due to preconditioning, both of these two terms are improved by $\mathcal{O}\left( \frac{\bar{\omega}^2}{\epsilon}  \right)$.

Rescaling mainly affects the complexity of QTA, which is almost linear with respect to the reciprocal of the accuracy of QTA. The accuracy of QTA is determined by the convergence condition and the residual of \eqref{LOsys: mnes}, see Lemma \ref{LOlemma: convergence condition}. In \cite{mohammadisiahroudi2022efficient}, they solve \eqref{LOsys: mnes} to get $z$ directly. They then analyze the residual of \eqref{LOsys: mnes} to determine the required accuracy of QTA, which is linear in the reciprocal of the condition number of \eqref{LOsys: mnes}.

In our algorithm, we solve \eqref{LOsys: RP-RAS} to get $\Delta y$ and then plug $\Delta y$ into \eqref{LOsys: mnes} using \eqref{LOeq: delta y and z}, which gives a linear system different from, but similar to \eqref{LOsys: mnes}. Similar analysis shows that the required accuracy of our QTA is linear in the reciprocal of the condition number of the different linear system, whose condition number is better than \eqref{LOsys: mnes} by $\mathcal{O}\left(\frac{\bar{\omega}^2}{\epsilon} \right)$.
The total improvement is the product of the improvements from these different sources.

\section{Conclusion}\label{sec: conclusion}

To mitigate the negative implication that the complexity of QLSA has linear dependence on condition number, and that the condition number of the Newton systems in II-QIPMs tend to infinity as optimality is approached, a preconditioned II-QIPM is proposed in this work. 
With the preconditioning method introduced by Chai and Toh in \cite{chai2007preconditioning}, we prove that the condition numbers in II-QIPMs are improved from $\mathcal{O}(1/\mu^2)$ to $\mathcal{O}(1/\mu)$. In each II-QIPM iteration, {the computational complexity dependence on $n$ comparable to the $\mathcal{O}(n^2)$ in \cite{mohammadisiahroudi2022efficient} and comparable to the {computational complexity of matrix-vector multiplication} in conjugate gradient methods for solving linear systems when $\|A\|_F=\mathcal{O}(\|A\|_2)$}.  
The complexity dependence on $\epsilon$ is $\mathcal{O}(1/\epsilon^6)$, which is worse than most classical II-IPMs but is quadratically better than the II-QIPM in \cite{mohammadisiahroudi2022efficient}. It is possible to further improve the dependence on $\epsilon$ by using the iterative refinement methods as discussed in \cite{mohammadisiahroudi2022efficient} but it is out of the scope of this paper.

\section*{Declarations}

\begin{itemize}
\item Funding: This work was supported by Defense Advanced Research Projects Agency as part of the project W911NF2010022: {\em {The Quantum
Computing Revolution and Optimization: Challenges and Opportunities}.} {This work was also supported by National Science Foundation CAREER DMS-2143915.} {This research also used resources of the Oak Ridge Leadership Computing Facility, which is a DOE Office of Science User Facility supported under Contract DE-AC05-00OR22725.}

\item Competing interests: The corresponding author is an editor for the journal Computation Optimization and Applications.
\item Ethics approval: Not applicable.
\item Consent to participate: Not applicable.
\item Consent for publication: Not applicable.
\item Availability of data and materials: Not applicable.
\item Code availability: Not applicable.
\item Authors' contributions: Not applicable.
\end{itemize}

\begin{appendices}

\section{}\label{secA1}
The following lemma is a special case of Theorem 4 from \cite{smith1992some} by considering positive semi-definite matrices instead of Hermitian matrices.
\begin{lemma}[Special case of Theorem 4 in \cite{smith1992some}]\label{lemma: interlace}
Consider two symmetric positive semi-definite matrices $(M_1, M_2) \in \mathbb{R}^{n\times n}\times \mathbb{R}^{n\times n}$, and let ${\rm rank}(M_2)=k$. Then their singular values satisfy
\begin{enumerate}
    \item $\sigma_i(M_1 + M_2)\geq \sigma_i(M_1)$, for $i=1,2,\dots,n$, and
    \item $\sigma_{i+k}(M_1 + M_2) \leq \sigma_i(M_1)$, for $i=1,2,\dots, n-k$.
\end{enumerate}
\end{lemma}

For block matrices, we have the following lemma as a special case of Theorem 5 in \cite{smith1992some} by considering $i=1$ and $i=n-r$.
\begin{lemma}\label{lemma: interlace block matrix}
Let $M$ be a symmetric positive definite matrix with $2\times 2$ block structure,
\begin{equation*}
    \begin{aligned}
        M = \begin{bmatrix}
        M_{11} & M_{12}\\M_{21}&M_{22}
        \end{bmatrix},
    \end{aligned}
\end{equation*}
where $M_{11}$ and $M_{22}$ are nonsingular. Then, the singular values of $M$ and the singular values of its Schur complement $(M/M_{22})$ satisfy
\begin{equation*}
    \begin{aligned}
    \sigma_{0}(M)\leq \sigma_{0}(M/M_{22})\leq \sigma_{1}(M/M_{22})\leq \sigma_{1}(M).
\end{aligned}
\end{equation*}
\end{lemma}

From Example 5.6.6 in \cite{horn2012matrix}, it is known that the largest singular value of any matrix coincides with its spectral norm. So the following lemma is equivalent to the submultiplicativity of spectral norm of matrices \cite{horn2012matrix}.
\begin{lemma}\label{lemma: largest singular value of matrix product}
Let $M_1\in \mathbb{R}^{m\times n}$ and $M_2 \in \mathbb{R}^{n \times l}$, then the largest singular value of $M_1M_2$ is bounded from above by the largest singular values of $M_1$ and $M_2$, i.e.,
\begin{equation*}
    \sigma_1(M_1M_2) \leq \sigma_1(M_1) \sigma_1(M_2).
\end{equation*}
\end{lemma}

Below we prove two lemmas involving the smallest nonzero singular value.
\begin{lemma}\label{lemma: least nonzero singular value of matrix product}
Let $M_1 \in \mathbb{R}^{m\times n}$ be a nonzero matrix and $M_2\in \mathbb{S}_{++}^{n\times n}$ be a diagonal matrix. The smallest nonzero singular value of $M_1M_2$ is bounded by the smallest nonzero singular value of $M_1$ and the smallest singular value, i.e., diagonal element, of $M_2$:
\begin{equation*}
    \sigma_0(M_1M_2) \geq \sigma_0(M_1) \sigma_0(M_2).
\end{equation*}
\end{lemma}
\begin{proof}
Without loss of generality, we can assume $M_1M_2$ has $k$ nonzero singular values:
\begin{equation*}
    \sigma_0(M_1M_2) = \sigma_k(M_1M_2).
\end{equation*}
Let $\mathbb{S}\subseteq \mathbb{R}^n$ be a subspace.
According to Courant-Fischer Min-Max theorem \cite{horn2012matrix}, it follows that
\begin{align*}
    \sigma_k(M_1M_2) &= \sqrt{\min_{\mathbb{S}:\dim(\mathbb{S}) = m-k+1} \max_{x\in \mathbb{S}\setminus{\{0\}}}\frac{x^\top M_1M_2(M_1M_2)^\top x}{x^\top x} }.
\end{align*}
It can be rewritten as
\begin{equation*}
    \sigma_k(M_1M_2) = \sqrt{\min_{\mathbb{S}:\dim(\mathbb{S}) = m-k+1} \max_{x\in \mathbb{S}\setminus{\{0\}}} \frac{(M_1^\top x)^\top M_2M_2^\top M_1^\top x}{(M_1^\top x)^\top M_1^\top x} \frac{(M_1^\top x)^\top M_1^\top x}{x^\top x} }.
\end{equation*}
Notice that
\begin{equation*}
    \frac{(M_1^\top x)^\top M_2M_2^\top M_1^\top x}{(M_1^\top x)^\top M_1^\top x} \geq \sigma_0(M_2M_2^\top)
\end{equation*}
as long as $M_1^\top x \neq 0$, which is true for $\sigma_k(M_1M_2)$ since $M_1M_2$ has $k$ nonzero singular values. So
\begin{align*}
    \sigma_k(M_1M_2) &\geq \sqrt{\min_{\mathbb{S}:\dim(\mathbb{S}) = m-k+1} \max_{x\in \mathbb{S}\setminus{\{0\}}} \frac{(M_1^\top x)^\top M_1^\top x}{x^\top x} \sigma_0(M_2M_2^\top)}\\
                 &= \sigma_k(M_1)\sigma_0(M_2).
\end{align*}
On one hand, if $M_1$ has more than $k$ nonzero singular values, then we can prove that $M_1M_2$ has more than $k$ nonzero singular values, which violates the assumption. On the other hand, if $\sigma_k(M_1) = 0$, then we have found an $x\in \mathbb{S}\setminus{\{0\}}$ with $\dim(\mathbb{S}) = m-k$ such that $M_1^\top x = 0$, which means $\sigma_0(M_1M_2) = \sigma_k(M_1M_2) = 0$ and violates the definition of $\sigma_0(M_1M_2)$. So $\sigma_0(M_1) = \sigma_k(M_1)$ and $\sigma_0(M_1M_2)\geq \sigma_0(M_1) \sigma_0(M_2)$.
\end{proof}

\begin{lemma}\label{lemma: nonzero singular value of block matrix with zero wrap}
Let $M_1\in \mathbb{S}_{++}^n$ and $M_2\in \mathbb{S}_{(n+l)\times (n+l)}$ with $l\in \mathbb{Z}_{+}$ and 
$$
M_2 = \begin{bmatrix}
       M_1 & 0\\
       0 & 0
\end{bmatrix}.
$$ Then the smallest nonzero singular value of $M_2$ equals to the smallest nonzero singular value of $M_1$, i.e., 
\begin{equation*}
    \sigma_0(M_1) = \sigma_0(M_2),
\end{equation*}
and the largest singular value of $M_2$ equals to the largest singular value of $M_1$, i.e.,
\begin{equation*}
    \sigma_1(M_1) = \sigma_1(M_2).
\end{equation*}
\end{lemma}
\begin{proof}
The results are straightforward by checking the SVD of $M_1$ and $M_2$.
\end{proof}

\end{appendices}


\bibliographystyle{abbrv} 
\bibliography{bib}

\end{document}